\documentclass[]{article}

\usepackage{graphicx}
\usepackage{amsmath}
\usepackage{amssymb}
\usepackage{amsthm}
\usepackage{pxfonts}
\usepackage{enumerate}
\usepackage{color}
\usepackage{mathdots}
\usepackage{sectsty}
\usepackage[hidelinks]{hyperref}
\usepackage{tikz}
\usepackage{caption}
\usepackage{adjustbox}
\usepackage{float}
\allowdisplaybreaks

\sectionfont{\scshape\centering\fontsize{11}{14}\selectfont}
\subsectionfont{\scshape\fontsize{11}{14}\selectfont}
\usepackage{fancyhdr}

\newcommand\shorttitle{On the moments of the partition function of the C$\beta$E field}
\newcommand\authors{Theodoros Assiotis}

\newtheorem{thm}{Theorem}[section]
\newtheorem{cor}[thm]{Corollary}
\newtheorem{lem}[thm]{Lemma}
\newtheorem{defn}[thm]{Definition}
\newtheorem{rmk}[thm]{Remark}
\newtheorem{prop}[thm]{Proposition}

\title{\large \bf ON THE MOMENTS OF THE PARTITION FUNCTION OF THE C$\beta$E FIELD}
\author{\small THEODOROS ASSIOTIS}
\date{}

\begin{document}

\maketitle

\begin{abstract}
We obtain a combinatorial formula for the positive integer moments of the partition function of the $C\beta E_{N}$ field, or equivalently the moments of the moments of the characteristic polynomial of the $C\beta E_{N}$ ensemble. We then use this formula to establish the large $N$ asymptotics of these moments in the ``moment-supercritical" regime. A key role is played by Jack polynomials.
\end{abstract}

\section{Introduction}
\subsection{The main result}
Let $\mathbb{T}$ be the unit circle and consider for $\beta>0$ the following probability measure on $\mathbb{T}^N$, with $\theta_j\in [0,2\pi)$:  
\begin{align}\label{EnsembleDef}
\frac{\Gamma\left(1+\frac{\beta}{2}\right)^N}{(2\pi)^N\Gamma \left(1+N\frac{\beta}{2}\right)}\prod_{1\le j<k \le N}\big|e^{\i \theta_j}-e^{\i \theta_k}\big|^{\beta}d\theta_1d\theta_2\cdots d\theta_N.
\end{align}
This probability measure is called the $C\beta E_{N}$ ensemble and we denote the expectation with respect to it by $\mathbb{E}_N^{(\beta)}$. For $\beta=2$ this measure is the law of the eigenvalues of a random Haar distributed $N\times N$ unitary matrix and it is called the CUE (unitary) ensemble. The $\beta=1$ and $\beta=4$ cases are also distinguished and are called the COE (orthogonal) and CSE (symplectic) ensembles respectively. In the orthogonal case one can obtain such a random matrix by taking $\mathbf{U}^T\mathbf{U}$ where $\mathbf{U}$ is an $N\times N$ Haar distributed unitary matrix, while the symplectic case is slightly more complicated, see \cite{KillipNenciu}, \cite{Matsumoto}.  Matrix models having the $C\beta E_N$ ensemble as the induced law of eigenvalues also exist for general $\beta>0$, see for example \cite{KillipNenciu}, but we will not need to use this fact here.

We now define the characteristic polynomial of the $C\beta E_N$ ensemble, with $\mathbf{z}=(z_1,\dots,z_N) $ where the $z_j=e^{\i \theta_j}$, $j=1,\dots,N$ are distributed according to (\ref{EnsembleDef}):
\begin{align}
\mathsf{\Psi}^{(N)}_{\mathbf{z}}(t)=\prod_{j=1}^N\left(1-e^{-\i t}z_j\right)=\prod_{j=1}^N\left(1-e^{-\i (t-\theta_j)}\right), \ t \in [0,2\pi].
\end{align}
We call $\log \mathsf{\Psi}_{\mathbf{z}}^{(N)}(\cdot)$ the $C\beta E_{N}$ field. Borrowing the statistical mechanics terminology from \cite{FyodorovKeating}, \cite{GnutzmannFyodorovKeating} we also define the partition function:
\begin{align}
\mathcal{Z}_\mathbf{z}^{(N)}(q)=\frac{1}{2\pi}\int_0^{2\pi}e^{2q\Re\log \mathsf{\Psi}_{\mathbf{z}}^{(N)}(t)}dt=\frac{1}{2\pi}\int_0^{2\pi}\big|\mathsf{\Psi}^{(N)}_{\mathbf{z}}(t)\big|^{2q}dt.
\end{align}
Finally, we define the moments of the partition function of the $C\beta E_{N}$ field, or using the terminology of \cite{BaileyKeating}, \cite{AssiotisKeating}, \cite{AssiotisBaileyKeating}, \cite{BaileyKeatingZeta} (from where we also borrow the notation) the moments of the moments of the characteristic polynomial of the $C\beta E_{N}$ ensemble:
\begin{align}
 \textnormal{MoM}_N^{(\beta)}(k;q)=\mathbb{E}_N^{(\beta)}\left[\left(\mathcal{Z}_\mathbf{z}^{(N)}(q)\right)^k\right]=\mathbb{E}_N^{(\beta)}\left[\left(\frac{1}{2\pi}\int_0^{2\pi}\big|\mathsf{\Psi}^{(N)}_{\mathbf{z}}(t)\big|^{2q}dt\right)^k\right].  
\end{align}

In this paper we give a combinatorial formula for $\textnormal{MoM}_N^{(\beta)}(k;q)$ for $k,q\in \mathbb{N}$ and general $\beta>0$ in Proposition \ref{CombinatorialRepMoM} below. We then use this formula to establish the large $N$ asymptotics of these moments in the so-called ``moment-supercritical" regime. This terminology comes from a connection to log-correlated Gaussian fields and multiplicative chaos that we briefly recall below, see \cite{KeatingWong} for more details.

\begin{thm}\label{MainResult} Let $k,q\in \mathbb{N}$ and $\beta>0$. If moreover $\beta$ satisfies:
\begin{itemize}
    \item $\beta<4q^2$, for $k=2$,
    \item $\beta \le 2$, for $k\ge 3$,
\end{itemize}
then we have the following asymptotics:
\begin{align}\label{AsymptoticsDisplay}
\lim_{N\to \infty}\frac{1}{N^{\frac{2}{\beta}(kq)^2-(k-1)}}\textnormal{MoM}_N^{(\beta)}(k;q)=\mathfrak{c}^{(\beta)}(k;q),
\end{align}
where the coefficient $\mathfrak{c}^{(\beta)}(k;q)$ is finite and strictly positive and is given as an integral of an explicit non-negative weight over continuous interlacing arrays with constraints, see (\ref{ConstIntegralRep}) for the precise definition.
\end{thm}

As far as we are aware, the result, even the precise order of the asymptotic in $N$, is new for parameters satisfying both $\beta\neq 2$ and $k\neq 1$ and we elaborate on the history and approaches to this problem below. The restriction to $\beta\le 2$ for $k\ge 3$ is a technical one and we expect the statement of the theorem to hold for all $\beta<2kq^2$ when $k>1$. Observe that, for $\beta<2kq^2$ the exponent of $N$ in (\ref{AsymptoticsDisplay}) is strictly greater than one. When $\beta=2kq^2$ this exponent becomes one, however it is expected when $k>1$, from the connection to multiplicative chaos \cite{KeatingWong} recalled shortly, that there should be a phase transition and that, up to a multiplicative constant, $\textnormal{MoM}_N^{(2kq^2)}(k;q)$ should grow to leading order like $N\log N$.

The weight mentioned in Theorem \ref{MainResult} is identically 1 when $\beta=2$ and $\mathfrak{c}^{(2)}(k;q)$ recovers the volume of the set $\mathsf{I}_c(k;q)$ from Definition \ref{DefContSet} below. For general $\beta>0$ it is very closely related to the orbital beta process, a certain probability distribution on continuous interlacing arrays, see \cite{GorinMarcus}, \cite{CuencaOrbital}, \cite{AssiotisNajnudel}. We note that while the integral expression (\ref{ConstIntegralRep}) for $\mathfrak{c}^{(\beta)}(k;q)$ is unambiguously defined for all $\beta>0$, it is infinite when $k\ge 2$ and $\beta$ is large enough (for reasons that we explain in Section \ref{PropAsympt2Sec}). This motivates the definition, for any fixed $k,q\in \mathbb{N}$, of the following subset of $(0,\infty)$:
\begin{align}
 \mathcal{A}(k;q)=\big\{\beta>0: \mathfrak{c}^{(\beta)}(k;q)<\infty \big\}.
\end{align}

Then, Theorem \ref{MainResult} is a consequence of the following two results, which use as a starting point the combinatorial formula for $\textnormal{MoM}_N^{(\beta)}(k;q)$ from Proposition \ref{CombinatorialRepMoM}. Proposition \ref{PropAsymptotics1} is proven in Section \ref{PropAsympt1Prop} while Proposition \ref{PropAsymptotics2} is proven in Section \ref{PropAsympt2Sec} along with a number of results on the leading order coefficient $\mathfrak{c}^{(\beta)}(k;q)$, including more explicit expressions for $k=1$ and $k=2$. In the special case $\beta=2$ corresponding to the CUE, $\mathfrak{c}^{(2)}(2;q)$ is known to have connections to integrable systems, in particular the Painlev\'{e} V equation and we discuss this briefly in Section \ref{IntSys}.

\begin{prop}\label{PropAsymptotics1}
Let $k,q\in \mathbb{N}$ and $\beta \in \mathcal{A}(k;q)$. Then, the asymptotics (\ref{AsymptoticsDisplay}) hold.
\end{prop}

\begin{prop}\label{PropAsymptotics2}
Let $k,q \in \mathbb{N}$. Then, we have:
\begin{itemize}
    \item $\mathcal{A}(1;q)=(0,\infty)$,
    \item $\mathcal{A}(2;q)=(0,4q^2)$,
    \item For $k\ge 3$, $(0,2] \subset \mathcal{A}(k;q)$ and moreover $[2kq^2,\infty)\cap\mathcal{A}(k;q)=\emptyset$.
\end{itemize}
\end{prop}

The fact that $[2kq^2,\infty)\cap\mathcal{A}(k;q)=\emptyset$ for $k>1$ is consistent with heuristics, explained right after, based on a connection to the theory of Gaussian multiplicative chaos, see \cite{KeatingWong} for the details. We expect that $\mathcal{A}(k;q)=(0,2kq^2)$ for all $k>1$ and we also give some brief heuristic arguments in support of this in Section \ref{PropAsympt2Sec}.

\subsection{Predictions from the connection to Gaussian log-correlated fields}\label{HeuristicsSection}

We now briefly recall the connection between $\mathsf{\Psi}^{(N)}_{\mathbf{z}}$ and  Gaussian log-correlated fields and multiplicative chaos, see the introduction of \cite{KeatingWong} for more details and precise statements and also \cite{Johansson}, \cite{HKO}, \cite{Lambert}, \cite{ChhaibiNajnudel}, \cite{JiangMatsumoto}, \cite{Webb}, \cite{GMCL1} for more on this topic. This connection begins with the following result, see \cite{Johansson}, \cite{HKO}, \cite{Lambert}, \cite{ChhaibiNajnudel}, \cite{JiangMatsumoto}  for the precise convergence statement:
\begin{align}\label{FieldConvergence}
    \log \big|\mathsf{\Psi}^{(N)}_{\mathbf{z}}(\cdot)\big| \overset{\textnormal{d}}{\longrightarrow} \frac{1}{\sqrt{\beta}}\mathsf{G}(\cdot), \ \ \textnormal{ as } N\to \infty,
\end{align}
where $\mathsf{G}$ is the Gaussian free field on the unit circle $\mathbb{T}$ with covariance:
\begin{align*}
    \mathbb{E}\left[\mathsf{G}(t)\mathsf{G}(s)\right]=-\log \big|e^{\i t}-e^{\i s}\big|, \ \ \forall s,t\in [0,2\pi).
\end{align*}
Now, it is possible to define, for a parameter $\gamma$ with $\gamma^2<2$, a non-trivial random measure on $\mathbb{T}$ called the Gaussian multiplicative chaos (GMC) associated to $\mathsf{G}$ which is written formally\footnote{Some care is needed since $\mathsf{G}$ is a random generalised function and we cannot just exponentiate it.} as, see for example \cite{Berestycki} for the details:
\begin{align*}
    \textnormal{GMC}_{\gamma}(dt)= \frac{e^{\gamma \mathsf{G}(t)}}{\mathbb{E}\left[e^{\gamma\mathsf{G}(t)}\right]}dt=e^{\gamma \mathsf{G}(t)-\frac{\gamma^2}{2}\mathbb{E}\left[\mathsf{G}^2(t)\right]} dt.
\end{align*}
From (\ref{FieldConvergence}) one might expect that we have the following convergence in law with respect to the topology of weak convergence of measures on $\mathbb{T}$, where $\gamma=\frac{2q}{\sqrt{\beta}}$:
\begin{align}\label{GMCconvergence}
 \frac{\big|\mathsf{\Psi}^{(N)}_{\mathbf{z}}(t)\big|^{2q}}{\mathbb{E}_N^{(\beta)}\left[\big|\mathsf{\Psi}^{(N)}_{\mathbf{z}}(t)\big|^{2q}\right]} dt \xrightarrow{N\rightarrow\infty} \textnormal{GMC}_{\gamma}(dt).
\end{align}
This convergence has been proven for $\beta=2$ in \cite{Webb}, \cite{GMCL1}. It is a very interesting and challenging task to extend this result to $\beta\neq 2$ but as far as we are aware this is still an open problem for any $\beta \neq 2$. However, see \cite{Lambert} for the analogous result for a small mesoscopic regularisation of $\mathsf{\Psi}^{(N)}_{\mathbf{z}}$, where also the so-called ``freezing transition" for the partition function $\mathcal{Z}_\mathbf{z}^{(N)}(q)$ is proven.

The total mass of the chaos $\textnormal{GMC}_{\gamma}\left(\mathbb{T}\right)$ is known to be an explicit random variable from the work of Remy \cite{Remy} that establishes a conjecture of Fyodorov and Bouchaud \cite{FyodorovBouchaud}. Its moments, when they exist, are also completely explicit and given in terms of Gamma functions, see \cite{Remy}, \cite{FyodorovBouchaud}. If $k$ is sufficiently small, equivalently $\beta$ is sufficiently large, $2kq^2<\beta$ so that the $k$-th moment of $\textnormal{GMC}_{\gamma}\left(\mathbb{T}\right)$ exists one might expect that (\ref{GMCconvergence}) can be extended to a convergence of the $k$-th moment of the total mass of the left hand side of (\ref{GMCconvergence}) to $\mathbb{E}\left[\textnormal{GMC}_{\gamma}\left(\mathbb{T}\right)^k\right]$. Using rotation invariance\footnote{By rotation invariance of $C\beta E_N$ we have that $\mathbb{E}_N^{(\beta)}\left[\big|\mathsf{\Psi}^{(N)}_{\mathbf{z}}(t)\big|^{2q}\right]=\textnormal{MoM}_N^{(\beta)}(1;q), \ \forall t\in [0,2\pi)$. Moreover, the leading order coefficient in the asymptotics of this quantity is completely explicit as we see next.} of the $C\beta E_N$, the above conjectural convergence of moments would imply that, in this ``moment-subcritical" regime, $\textnormal{MoM}_N^{(\beta)}(k;q)$ should grow to leading order like $N^{\frac{2}{\beta}kq^2}$ times an explicit constant, see \cite{KeatingWong} for more details.  Although not written out explicitly, for $\beta=2$ and a restricted range of parameters $k,q$, see \cite{KeatingWong} for the details, this convergence of moments is a direct consequence of the proofs in \cite{Webb}, \cite{GMCL1}. At ``moment-criticality" $\beta=2kq^2$ with\footnote{This is because we need $\gamma^2=\frac{4q^2}{\beta}<2$.} $k>1$ a more refined GMC heuristic developed in \cite{KeatingWong} gives a conjecture that the moments of moments should grow to leading order like $N\log N$ times an explicit, albeit more involved compared to the moment-subcritical case, constant. For $\beta=2$ and $k\in \mathbb{N}$ this conjecture was proven in the same paper \cite{KeatingWong}.

In the ``moment-supercritical" regime $\beta<2kq^2$ that we study in this paper the connection to the moments of the GMC breaks down and we get a completely different asymptotic behaviour which is much less understood. In fact, some more involved GMC heuristics, see \cite{KeatingWong}, can still predict the correct power of $N$ in the asymptotic but not the leading order coefficient. Although the leading order coefficient in this regime is in general not as explicit (and in fact we do not expect it to be) as in the moment-subcritical and critical regimes it has some non-trivial structure which leads to a representation, in the simplest\footnote{To be precise this is the simplest possible case which is not completely explicit as $k=1$.} possible case when $\beta=2$ and $k=2$, in terms of a Painlev\'{e} V transcedent. We believe that more intricate connections to the theory of integrable systems should exist beyond this simplest possible case and the results of this paper could be used as a starting point for investigating this. It is interesting that the combinatorial formula for $\textnormal{MoM}_N^{(\beta)}(k;q)$ we give in Proposition \ref{CombinatorialRepMoM} also exists for parameters in the moment critical and subcritical regimes. This is a new feature of the $C\beta E_N$ for general $\beta$ that is not present for the special case of CUE studied previously. In particular, it might be possible that this formula can be used to access the asymptotics of $\textnormal{MoM}_N^{(2kq^2)}(k;q)$ for the critical case $\beta=2kq^2$ with $k,q\in \mathbb{N}$ and $k>1$. However, the arguments will be different, and more involved, than the ones presented here and so we do not pursue it further in this paper.

\subsection{Some history and motivation}\label{HistorySection}
For $k=1$ and general $\beta>0$, Theorem \ref{MainResult} is originally due to Keating and Snaith from \cite{KeatingSnaith} using the Selberg integral and asymptotics for the Barnes G-function. The leading order coefficient is in fact completely explicit in this case:
\begin{align}\label{k=1explicit}
\mathfrak{c}^{(\beta)}(1;q)=\prod_{i=1}^q\frac{\Gamma\left(\frac{2}{\beta}i\right)}{\Gamma\left(\frac{2}{\beta}(q+i)\right)}.   
\end{align}
An alternative proof using symmetric function theory is due to Matsumoto from \cite{Matsumoto}. Yet another proof making use of spectral theory and orthogonal polynomials on the unit circle can be obtained from \cite{BourgadeNikeghbaliRouault}. As we will see in Lemma \ref{ExplicitExpressions} below $\mathfrak{c}^{(\beta)}(1;q)$ also arises as an integral of a special weight over a continuous interlacing array.

In the special case of the CUE, namely $\beta=2$, which has a lot of extra structure, there has been a great deal of work on the problem of asymptotics of moments of moments. In particular, for $k=2$ and real $q$ the theorem above is a consequence of the results of Claeys and Krasovsky from \cite{ClaeysKrasovsky} on asymptotics of Toeplitz determinants with merging singularities who also go on to prove that $\mathfrak{c}^{(2)}(2;q)$ has a representation in terms of the Painlev\'e V equation. This is achieved using the Riemann-Hilbert problem method. The Riemann-Hilbert problem analysis is also the key technical tool in proving the convergence (\ref{GMCconvergence}) to GMC for $\beta=2$ \cite{Webb}, \cite{GMCL1} and in establishing in the moment-subcritical \cite{Webb}, \cite{GMCL1} and moment-critical regimes \cite{KeatingWong} the convergence of moments of moments (again for $\beta=2$). Later two alternative proofs of the asymptotics for $k=2$ and $q\in \mathbb{N}$ were given in \cite{KeatingRodgersRodittyGershonRudnick}. One of them using multiple contour integrals from \cite{Autocorrelation} and the other using symmetric function theory, based on results of Bump and Gamburd \cite{BumpGamburd}, and two different expressions for $\mathfrak{c}^{(2)}(2;q)$ were obtained. Using the expression of $\mathfrak{c}^{(2)}(2;q)$ coming from symmetric function theory a different proof of the connection to Painlev\'e V was then given in \cite{BasorGeRubinstein}. The complex analytic approach using multiple contour integrals was extended in \cite{BaileyKeating} to establish the asymptotics for general $k,q\in \mathbb{N}$. Afterwards, the combinatorial approach involving symmetric functions was also extended to $k,q\in \mathbb{N}$ in \cite{AssiotisKeating} and a different, geometric expression for $\mathfrak{c}^{(2)}(k;q)$ given as a volume\footnote{Namely the volume of the set $\mathsf{I}_c(k;q)$ from Definition \ref{DefContSet} below.} was obtained. The combinatorial approach to this problem essentially boils down to counting lattice points in certain complicated regions\footnote{The lattice points are essentially given by discrete interlacing arrays with constraints.}. This approach was also then adapted in \cite{AssiotisBaileyKeating} to deal with the corresponding question of asymptotics of moments of moments in the more involved case of Haar distributed $Sp(2N)$ and $SO(2N)$ matrices\footnote{Not to be confused with the symplectic CSE and orthogonal COE ensembles mentioned earlier.}. Recently, Fahs \cite{Fahs} using Riemann-Hilbert problem techniques was able to establish the order of the asymptotics in $N$ for general $k\in \mathbb{N}$ and real $q$, however an expression for $\mathfrak{c}^{(2)}(k;q)$ for $k\ge 3$ and non-integer $q$ is still lacking in the moment-supercritical regime. As far as we know, no rigorous results are available at present in the moment-supercritical regime when we also allow for non-integer values of $k$.

Now, for general $\beta\neq 2$, as far as we are aware, there were no rigorous results (including for the moment subcritical and critical regimes) for $k\neq 1$ prior to the present work. In terms of available approaches to this question, the Riemann-Hilbert problem techniques do not apply\footnote{This lack of a key analytic technical tool is also one of the main reasons why (\ref{GMCconvergence}) has not been established yet for $\beta\neq 2$.} since we are no longer in the determinantal setting of $\beta=2$ and similarly we are not aware of any multiple contour integral formulae for general $\beta\neq 2$. Moreover, an approach using random orthogonal polynomials on the unit circle as in \cite{BourgadeNikeghbaliRouault} does not seem well-adapted to this problem. However, an approach based on symmetric function theory and combinatorics does work. Our starting point is a result of Matsumoto \cite{Matsumoto} which connects expectations of products of characteristic polynomials from the C$\beta$E to the Jack symmetric polynomials. The main difference compared to the $\beta=2$ case from \cite{AssiotisKeating} is that now instead of simply counting lattice points we also include a weight which comes from the combinatorial formula for Jack polynomials and makes the analysis more complicated\footnote{We also streamline the exposition of some preliminary combinatorial results for the moments of moments compared to \cite{AssiotisKeating}, for example we go directly to discrete interlacing arrays without passing through Young diagrams first.}. Even with this extra complication, the proof is still relatively short and this can be viewed as a testament to the efficiency of the method. We should also point out that this is of course not the first time Jack polynomials have been used in answering asymptotic questions related to the C$\beta$E, see for example \cite{JiangMatsumoto}, \cite{ChhaibiNajnudel}, \cite{NajnudelPaquetteSimm}. However, both the questions and also the way Jack polynomials were used to answer them in these works are different from what we do here. It would be very interesting to extend the asymptotics established in this paper beyond positive integer values of $k$ and $q$ but this would most likely require new ideas.

Finally, it is important to mention that the moments of moments are also closely related to conjectures by Fyodorov, Hiary and Keating on the extreme value theory of the $C\beta E_N$ field $\log \mathsf{\Psi}^{(N)}_{\mathbf{z}}$, see \cite{FyodorovKeating}, \cite{GnutzmannFyodorovKeating}, \cite{BaileyKeating} for more details and \cite{ArguinBeliusBourgade}, \cite{PaquetteZeitouni}, \cite{ChhaibiMadauleNajnudel} for rigorous progress on these conjectures and moreover for $\beta=2$ they are connected to the corresponding moments of moments of the Riemann zeta function, see \cite{BaileyKeatingZeta}.

Regarding future directions, the method followed here should also work for the $\beta$-ensemble versions of $Sp(2N)$ and $SO(2N)$ random matrices. Instead of the Jack polynomials one uses the Heckman-Opdam Jacobi polynomials, see Section 5 of \cite{Matsumoto}. However, both the combinatorics, as can be seen from the $\beta=2$ case in \cite{AssiotisBaileyKeating}, and also the weight are significantly more complicated and we do not pursue this further in this paper. We hope to return to this question in future work. It is worth mentioning that connections to Gaussian log-correlated fields and multiplicative chaos measures also exist in the setting of $Sp(2N)$ and $SO(2N)$ random matrices, see \cite{ForkelKeating}, \cite{KeatingWong} for more details. We expect that these results should extend to the corresponding general $\beta$-ensemble versions but such questions have not been explored yet. Finally, it would also be possible to apply the method presented in this paper, with more involved computations, at a higher level of symmetric functions, for the Macdonald weight, which involves two parameters; but in this case the definition of the moments of moments needs to be modified accordingly.

\paragraph{Acknowledgements} I would like to thank Jon Keating and Emma Bailey for introducing me in the first place to questions about moments of characteristic polynomials of random matrices. I am very grateful to Mo Dick Wong for pointing out to me the GMC heuristics for the moments of the C$\beta$E field partition function that are presented in Section \ref{HeuristicsSection} and for pointers to the literature. This led me realise that the integral defining the leading order coefficient is not always finite as incorrectly stated in the first version of the paper. I am also very grateful to two anonymous referees for a very careful reading of the paper and useful comments and suggestions.

\paragraph{Data availability} No data was used in the research described in this paper.

\section{A combinatorial representation for the moments for finite N} \label{CombinatoricsSection}
We first need a number of preliminaries from symmetric function theory. We write $\mathbb{Z}_+=\{0,1,2,3,\dots\}$.  We define the space of non-negative signatures of length $M\in \mathbb{N}$ by:
\begin{align*}
\mathsf{S}_+^{(M)}=\big\{(\lambda_1,\dots,\lambda_M)\in \mathbb{Z}_+^M:\lambda_1\ge \lambda_2 \ge \dots \ge \lambda_M \big\}.    
\end{align*}
For $\lambda\in \mathsf{S}_+^{(M)}$ we write $|\lambda|=\sum_{i=1}^M \lambda_i$. We say that $\mu\in \mathsf{S}_+^{(M)}$ and $\lambda \in \mathsf{S}_+^{(M+1)}$ interlace and write $\mu\prec \lambda$ if: 
\begin{align}\label{interlacingineq}
\lambda_1\ge \mu_1 \ge \lambda_2\ge \dots \ge \lambda_M\ge \mu_M\ge \lambda_{M+1}.   
\end{align}
We also call a sequence of signatures which interlace a discrete interlacing array. For $\mu\in\mathsf{S}_+^{(M)}, \lambda \in \mathsf{S}_+^{(M+1)}$ that interlace we define the following non-negative weight:
\begin{align}\label{weightdef}
 \psi_{\lambda/\mu}^{(\delta)}=\prod_{1\le i \le j \le M}  \frac{(\mu_i-\mu_j+\delta(j-i)+\delta)_{\mu_j-\lambda_{j+1}}(\lambda_i-\mu_j+\delta(j-i)+1)_{\mu_j-\lambda_{j+1}}}{(\mu_i-\mu_j+\delta(j-i)+1)_{\mu_j-\lambda_{j+1}}(\lambda_i-\mu_j+\delta(j-i)+\delta)_{\mu_j-\lambda_{j+1}}} ,
\end{align}
where $(t)_m=t(t+1)\cdots (t+m-1)=\frac{\Gamma(t+m)}{\Gamma(t)}$ is the Pochhammer symbol. Also, observe that $\psi_{\lambda/\mu}^{(1)}\equiv 1$.

\begin{defn}\label{JackPolyDef}
Let $\delta>0$ and $\lambda \in \mathsf{S}_+^{(M)}$. Then, we define the Jack polynomial indexed by $\lambda$ by the combinatorial formula\footnote{The sum is over all tuples of signatures $\left(\lambda^{(1)},\lambda^{(2)},\dots,\lambda^{(M-1)}\right)$ which satisfy $\lambda^{(1)}\prec \lambda^{(2)}\prec \dots \prec \lambda^{(M-1)} \prec\lambda$.}
\begin{align}
\mathcal{P}_{\lambda}\left(x_1,x_2,\dots,x_M;\delta\right)=\sum_{\lambda^{(1)}\prec \lambda^{(2)}\prec\dots \prec \lambda^{(M-1)}\prec \lambda^{(M)}=\lambda}  \psi^{(\delta)}_{\lambda/\lambda^{(M-1)}}\psi^{(\delta)}_{\lambda^{(M-1)}/\lambda^{(M-2)}}\cdots \psi^{(\delta)}_{\lambda^{(2)}/\lambda^{(1)}} \times\nonumber \\  x^{|\lambda|-|\lambda^{(M-1)}|}_M x_{M-1}^{|\lambda^{(M-1)}|-|\lambda^{(M-2)}|} \cdots x_2^{|\lambda^{(2)}|-|\lambda^{(1)}|} x_1^{|\lambda^{(1)}|}.
\end{align}
\end{defn}

\begin{rmk}
Here we use the notation conventions of Okounkov and Olshanski from \cite{OkounkovOlshanskiJack}, with our fixed parameter $\delta=\theta$ from their paper. Often in the literature, see \cite{Macdonald}, \cite{Stanley}, the inverse parameter $\alpha=1/\theta$ is used. The combinatorial definition above is a consequence of the branching rule for Jack polynomials, see Section 2.3 in \cite{OkounkovOlshanskiJack}. The Jack polynomials are in fact symmetric (although this is not immediately evident from the definition above) and are orthogonal with respect to the C$\beta$E weight, see \cite{Macdonald}, \cite{Stanley}, \cite{Matsumoto} for more details. Finally, in the special case $\delta=1$ the Jack polynomials specialise to the Schur polynomials: $\mathcal{P}_\lambda(\cdot;1)=\mathsf{s}_{\lambda}(\cdot)$. 
\end{rmk}

The following proposition due to Matsumoto, which is a special case of the results of Section 4.2 in \cite{Matsumoto}, will be our starting point. This proposition is the C$\beta$E generalization of the results of Bump and Gamburd \cite{BumpGamburd} on the CUE ($\beta=2$) characteristic polynomial which formed the starting point of the investigation in \cite{AssiotisKeating}.
\begin{prop}\label{MatsumotoProp}[Matsumoto \cite{Matsumoto}]
Let $\beta>0$ and $N,k,q\in \mathbb{N}$. Then, we have:
\begin{align*}
\mathbb{E}_N^{(\beta)}\left[\big|\mathsf{\Psi}^{(N)}_{\mathbf{z}}(t_1)\big|^{2q}\dots \big|\mathsf{\Psi}^{(N)}_{\mathbf{z}}(t_k)\big|^{2q}\right]=\frac{\mathcal{P}_{(N,\dots,N,0,\dots,0)}\left(e^{\i t_1},\dots, e^{\i t_1},e^{\i t_2},\dots,e^{\i t_2}, \dots, e^{\i t_k},\dots, e^{\i t_k};\frac{2}{\beta}\right)  }{\prod_{j=1}^k e^{\i N qt_j}}  
\end{align*}
where each variable $e^{\i t_j}$, for $j=1,\dots,k$ appears $2q$ times and $(N,\dots,N,0,\dots,0)\in \mathsf{S}_+^{(2kq)}$ consists of $kq$ $N$'s and $kq$ $0$'s.
\end{prop}

We will first obtain a preliminary combinatorial representation for $\textnormal{MoM}_N^{(\beta)}(k;q)$. Towards this end we begin by defining the following set $\mathsf{J}_N(k;q)$. It consists of tuples of non-negative signatures $\mathsf{\Lambda}=\left(\lambda^{(1)},\lambda^{(2)},\lambda^{(3)},\dots, \lambda^{(2kq-1)},\lambda^{(2kq)}\right)$ so that: $\lambda^{(i)}\in \mathsf{S}_+^{(i)}$,
\begin{align*}
 \lambda^{(1)}\prec \lambda^{(2)}\prec \lambda^{(3)}\prec \dots \prec \lambda^{(2kq-1)}\prec \lambda^{(2kq)},  
\end{align*}
 $\lambda^{(2kq)}=(N,\dots,N,0,\dots,0)$, where both the $N$'s
and the $0$'s appear $kq$ times, and moreover the following $(k-1)$ sum constraints\footnote{The sum constraint for $\lambda^{(2kq)}$ is automatically satisfied by definition.} are satisfied:
\begin{align*}
\big|\lambda^{(2jq)}\big|=\sum_{i=1}^{2jq}\lambda_i^{(2jq)}=Njq, \ \ j=1, \dots, k-1.
\end{align*}

\begin{prop}\label{PreliminaryRepresentationProp}
Let $\beta>0$ and $N,k,q\in \mathbb{N}$. Then, we have:
\begin{align}
\textnormal{MoM}_N^{(\beta)}(k;q)=\sum_{\mathsf{\Lambda}\in\mathsf{J}_N(k;q)}\psi_{\lambda^{(2kq)}/\lambda^{(2kq-1)}}^{\left(\frac{2}{\beta}\right)}\psi_{\lambda^{(2kq-1)}/\lambda^{(2kq-2)}}^{\left(\frac{2}{\beta}\right)}\cdots \psi_{\lambda^{(3)}/\lambda^{(2)}}^{\left(\frac{2}{\beta}\right)} \psi_{\lambda^{(2)}/\lambda^{(1)}}^{\left(\frac{2}{\beta}\right)}.       
\end{align}
\end{prop}

\begin{proof}
We first apply Fubini's theorem:
\begin{align*}
\textnormal{MoM}_N^{(\beta)}(k;q)=\left(\frac{1}{2\pi}\right)^{k}\int_0^{2\pi}\dots\int_0^{2\pi}\mathbb{E}_N^{(\beta)}\left[\big|\mathsf{\Psi}^{(N)}_{\mathbf{z}}(t_1)\big|^{2q}\dots \big|\mathsf{\Psi}^{(N)}_{\mathbf{z}}(t_k)\big|^{2q}\right]dt_1\dots dt_k.   
\end{align*}
Then, we use Proposition \ref{MatsumotoProp} and Definition \ref{JackPolyDef}. The result then easily follows from the fact that, with $c\in \mathbb{Z}$:
\begin{align*}
 \frac{1}{2\pi}\int_0^{2\pi}e^{\i c t}dt=\begin{cases}1, \ \ c=0,\\
 0, \ \ \textnormal{otherwise}.
\end{cases}
\end{align*}
\end{proof}

We will now obtain an alternative combinatorial representation for $\textnormal{MoM}_N^{(\beta)}(k;q)$ which is well-suited for taking the large $N$ limit. The following definition is a key ingredient.
\begin{defn} Let $N,k,q \in \mathbb{N}$. We define the set
$\mathsf{I}_N(k;q)$ as follows. It consists of tuples of non-negative signatures $\mathsf{\Lambda}=\left(\lambda^{(1)},\lambda^{(2)},\dots,\lambda^{(kq-1)},\lambda^{(kq)}=\tilde{\lambda}^{(kq)},\tilde{\lambda}^{(kq-1)},\dots, \tilde{\lambda}^{(2)},\tilde{\lambda}^{(1)}\right)$ so that: $\lambda^{(i)},\tilde{\lambda}^{(i)}\in \mathsf{S}_+^{(i)}$, 
\begin{align*}
\lambda^{(1)}\prec \lambda^{(2)} \prec \lambda^{(3)}\prec \dots \prec \lambda^{(kq)}=\tilde{\lambda}^{(kq)}\succ \dots \succ \tilde{\lambda}^{(3)} \succ \tilde{\lambda}^{(2)} \succ \tilde{\lambda}^{(1)},
\end{align*}
$\lambda_i^{(j)},\tilde{\lambda}_i^{(j)}\in \llbracket 0,N \rrbracket =\{0,1,\dots,N\}$ and finally the following $(k-1)$ sum constraints are satisfied:
\begin{align*}
 \big|\lambda^{(2jq)}\big|&=\sum_{i=1}^{2jq}\lambda_i^{(2jq)}=Njq, \ \ j=1,\dots,\big\lfloor \frac{k}{2}\big\rfloor,\\
  \big|\tilde{\lambda}^{(2jq)}\big|&=\sum_{i=1}^{2jq}\tilde{\lambda}_i^{(2jq)}=Njq, \ \ j=1,\dots,\big\lfloor \frac{k}{2}\big\rfloor .
\end{align*}
\end{defn}
Observe that, $\mathsf{I}_N(k;q)$ has $(kq)^2-(k-1)$ free (non-fixed) coordinates.

\begin{figure}

\scalebox{0.7}{
\begin{tikzpicture}

\draw[ultra thick, blue]   (4,-1) -- (0,4);
\draw[ultra thick,blue]   (4,9) -- (0,4);
\draw[ultra thick,blue]   (4,9) -- (8,4);
 \draw[ultra thick,blue]   (4,-1) -- (8,4);

 \draw [fill=gray!20] (0,4) -- (-4,9)-- (4,9)--(0,4); 
 
  \draw [fill=gray!20] (8,4) -- (12,9)-- (4,9)--(8,4); 
  \node[] at (0,7.5) {\LARGE \bf{0}};

     \node[] at (8,7.5) {\LARGE \bf{N}};

   \draw[ultra thick,red]   (2.4,7) -- (5.6,7);      

   \draw[ultra thick, dashed,red]   (2.4,7) -- (-2.4,7);      

   \draw[ultra thick, dashed,red]   (5.6,7) -- (10.4,7);

       \draw[ultra thick,red]   (2.4,1) -- (5.6,1);         
      \draw[ultra thick,red]   (0.8,3) -- (7.2,3);            
      
\draw[ultra thick,red]   (0.8,5) -- (7.2,5);     

   \draw[ultra thick, dashed,red]   (1.2,5) -- (-0.8,5);      

   \draw[ultra thick, dashed,red]   (6.8,5) -- (8.8,5);

\draw [<->] (4,0.8) -- (4,-0.8);
\node[right] at (4,0){$2q$};
\draw [<->] (4,1.2) -- (4,2.8);
\node[right] at (4,2){$2q$};
\draw [<->] (4,3.2) -- (4,4.8);
\node[right] at (4,4){$2q$};
\draw [<->] (4,5.2) -- (4,6.8);
\node[right] at (4,6){$2q$};
\draw [<->] (4,7.2) -- (4,8.8);
\node[right] at (4,8){$2q$};

\draw[thick] (4,4) ellipse (4.5cm and 0.4cm);
 \node[] at (2.5,4) {$\boldsymbol{\cdots}$};
  \node[] at (5.5,4) {$\boldsymbol{\cdots}$};
\draw[fill,green] (0.5,4) circle [radius=0.1];   
\draw[fill,green] (1,4) circle [radius=0.1];   
\draw[fill,green] (1.5,4) circle [radius=0.1];   
\draw[fill,green] (7.5,4) circle [radius=0.1];   
\draw[fill,green] (7,4) circle [radius=0.1];   
\draw[fill,green] (6.5,4) circle [radius=0.1];

\end{tikzpicture}

}

\caption{A figure showing the shaded triangles of fixed coordinates of $0$'s and $N$'s. This results in two discrete interlacing arrays joined at the top row as shown in the figure in green. The solid red lines correspond to the sum constraints in $\mathsf{I}_N(k;q)$, while their continuations involving the dashed red part correspond to the sum constraints in $\mathsf{J}_N(k;q)$; in the figure $k=5$ so we have $4$ such constraints.}\label{Figure}
\end{figure}

\begin{lem}\label{lemmabijection}
Let $N,k,q\in \mathbb{N}$. Then, there exists a bijection $\mathcal{S}$ between $\mathsf{J}_N(k;q)$ and $\mathsf{I}_N(k;q)$
\begin{align*}
 \mathcal{S}:\mathsf{J}_N(k;q)&\longrightarrow \mathsf{I}_N(k;q),
 \end{align*}
given in terms of the coordinates by:
\begin{align*}
 \lambda^{(i)}&\mapsto \lambda^{(i)}, \ \ i=1,\dots, kq-1\\
\lambda^{(kq)} &\mapsto \lambda^{(kq)}=\tilde{\lambda}^{(kq)}, \\
 \left(\lambda^{(i)}_{i-kq+1},\dots,\lambda^{(i)}_{kq} \right)&\mapsto \left(\tilde{\lambda}_1^{(2kq-i)}, \dots, \tilde{\lambda}_{2kq-i}^{(2kq-i)}\right), \ \ i=kq+1, \dots, 2kq-1.
\end{align*}
\end{lem}

\begin{proof}
The key observation is that the interlacing constraints and the form of the top row $\lambda^{(2kq)}=(N,\dots,N,0,\dots,0)$ fix two triangles of coordinates, one of them filled with $0$'s and the other one with $N$'s, see Figure \ref{Figure} for an illustration. This gives two discrete interlacing arrays joined at their respective top rows. The rest of the proof is essentially a relabelling of the free coordinates. Finally, it is easily seen that the sum constraints transform in the desired way. 
\end{proof}

Now we would like to understand how the weights $\psi$ transform under the bijection $\mathcal{S}$. For this it will be convenient to introduce some more notation. For any $M\in \mathbb{N}$ and $\lambda \in \mathsf{S}_+^{(M)}$ with $\lambda_1\le N$ we define:
\begin{align*}
\mathfrak{e}_N(\lambda)=(N,\lambda_1,\lambda_2,\dots,\lambda_M,0)\in \mathsf{S}_+^{(M+2)}.    
\end{align*}
By convention $\mathfrak{e}_N(\emptyset)=(N,0)$. Also, observe that if $\lambda \in \mathsf{S}^{(M)}_+$, $\nu \in \mathsf{S}_+^{(M+1)}$ with $\lambda \prec \nu$ and moreover $\lambda_1,\nu_1\le N$ then $\nu \prec \mathfrak{e}_N(\lambda)$.
\begin{lem}\label{lemmaweight}
Let $\beta>0$ and $N,k,q\in \mathbb{N}$. Then, under the bijection $\mathcal{S}$ between $\mathsf{J}_N(k;q)$ and $\mathsf{I}_N(k;q)$ we have
\begin{align*}
\psi^{\left(\frac{2}{\beta}\right)}_{\lambda^{(2kq-i)}/\lambda^{(2kq-i-1)}}=\psi^{\left(\frac{2}{\beta}\right)}_{\mathfrak{e}_N\left(\tilde{\lambda}^{(i)}\right)/\tilde{\lambda}^{(i+1)}}    , \ \ i=0,1,\dots,kq-1,
\end{align*}
where we use the convention $\tilde{\lambda}^{(0)}=\emptyset$.
\end{lem}
\begin{proof}
Direct computation by noting that when $\mu_i=\lambda_i$ or $\mu_j=\lambda_{j+1}$ then the factor corresponding to the indices $(i,j)$ in the product definition (\ref{weightdef}) of $\psi_{\lambda/\mu}^{(\delta)}$ is identically 1.
\end{proof}

\begin{prop}\label{CombinatorialRepMoM}
Let $\beta>0$ and $N,k,q\in \mathbb{N}$. Then, we have:
\begin{align}
 \textnormal{MoM}_N^{(\beta)}(k;q)=\sum_{\mathsf{\Lambda}\in\mathsf{I}_N(k;q)}\psi_{\mathfrak{e}_N(\emptyset)/\tilde{\lambda}^{(1)}}^{\left(\frac{2}{\beta}\right)}\psi_{\mathfrak{e}_N(\tilde{\lambda}^{(1)})/\tilde{\lambda}^{(2)}}^{\left(\frac{2}{\beta}\right)}\cdots\psi_{\mathfrak{e}_N(\tilde{\lambda}^{(kq-1)})/\lambda^{(kq)}}^{\left(\frac{2}{\beta}\right)}\psi_{\lambda^{(kq)}/\lambda^{(kq-1)}}^{\left(\frac{2}{\beta}\right)}\cdots  \psi_{\lambda^{(2)}/\lambda^{(1)}}^{\left(\frac{2}{\beta}\right)}.       
\end{align}
\end{prop}

\begin{proof}
This follows by combining Proposition \ref{PreliminaryRepresentationProp} along with Lemma \ref{lemmabijection} and Lemma \ref{lemmaweight}.
\end{proof}

\section{The large N limit}

\subsection{Proof of Proposition \ref{PropAsymptotics1}}\label{PropAsympt1Prop}
We begin with some preliminaries. We need to introduce the continuous analogues of the notions and objects we saw in the discrete setting of Section \ref{CombinatoricsSection}. For any $M\in \mathbb{N}$, we define the Weyl chamber with non-negative coordinates:
\begin{align*}
\mathsf{W}_+^{(M)}=\big\{(x_1,\dots,x_M)\in \mathbb{R}_+^M:x_1\ge x_2 \ge \dots \ge x_M \big\}.   
\end{align*}
For $x\in \mathsf{W}_+^{(M)}$ we write $|x|=\sum_{i=1}^M x_i$. We say that $y\in \mathsf{W}_+^{(M)}$ and $x\in \mathsf{W}_+^{(M+1)}$ interlace and still denote this by $y \prec x$ if the same inequalities (\ref{interlacingineq}) as in the discrete case hold. Also, similarly to the discrete setting we call a sequence of configurations which interlace a continuous interlacing array. We now define the continuous analogue of $\mathsf{I}_N(k;q)$.

\begin{defn}\label{DefContSet} Let $k,q \in \mathbb{N}$. We define the set
$\mathsf{I}_c(k;q)$ as follows. It consists of tuples $\mathsf{X}=\left(x^{(1)},x^{(2)},\dots,x^{(kq-1)},x^{(kq)}=\tilde{x}^{(kq)},\tilde{x}^{(kq-1)},\dots, \tilde{x}^{(2)},\tilde{x}^{(1)}\right)$ so that: $x^{(i)},\tilde{x}^{(i)}\in \mathsf{W}_+^{(i)}$, 
\begin{align*}
x^{(1)}\prec x^{(2)} \prec x^{(3)}\prec \dots \prec x^{(kq)}=\tilde{x}^{(kq)}\succ \dots \succ \tilde{x}^{(3)} \succ \tilde{x}^{(2)} \succ \tilde{x}^{(1)},
\end{align*}
$x_i^{(j)},\tilde{x}_i^{(j)}\in [0,1]$ and finally the following $(k-1)$ sum constraints are satisfied:
\begin{align*}
 \big|x^{(2jq)}\big|&=\sum_{i=1}^{2jq}x_i^{(2jq)}=jq, \ \ j=1,\dots,\left\lfloor \frac{k}{2}\right\rfloor,\\
  \big|\tilde{x}^{(2jq)}\big|&=\sum_{i=1}^{2jq}\tilde{x}_i^{(2jq)}=jq, \ \ j=1,\dots,\left\lfloor \frac{k}{2}\right\rfloor .
\end{align*}
\end{defn}
Observe that, as in the discrete setting, $\mathsf{I}_c(k;q)$ has $(kq)^2-(k-1)$ free (non-fixed) coordinates. We also define the continuous counterpart of $\mathfrak{e}_N$. For any $M\in \mathbb{N}$ and $x \in \mathsf{W}_+^{(M)}$ with $x_1\le 1$ we define:
\begin{align*}
\mathfrak{e}_c(x)=(1,x_1,x_2,\dots,x_M,0)\in \mathsf{W}_+^{(M+2)}.    
\end{align*}
By convention $\mathfrak{e}_c(\emptyset)=(1,0)$. As in the discrete case, observe that if $y \in \mathsf{W}^{(M)}_+$, $x \in \mathsf{W}_+^{(M+1)}$ with $y \prec x$ and moreover $x_1,y_1\le 1$ then $x \prec \mathfrak{e}_c(y)$.

For $y\in \mathsf{W}_+^{(M)}, x\in \mathsf{W}_+^{(M+1)}$ which interlace we define the following non-negative weight, which is the continuous analogue of $\psi_{\lambda/\mu}^{(\delta)}$:
\begin{align*}
\phi_{M,M+1}^{(\delta)}(y,x)&=\frac{1}{\Gamma(\delta)^M}\prod_{1\le i<j \le M}(y_i-x_{j+1})^{\delta-1}(y_i-y_j)^{1-\delta}(x_i-y_j)^{\delta-1}(x_i-x_{j+1})^{1-\delta}\\
&\times\prod_{i=1}^M(y_i-x_{i+1})^{\delta-1}(x_i-y_i)^{\delta-1}(x_i-x_{i+1})^{1-\delta}\\
&=\frac{1}{\Gamma(\delta)^M}\prod_{i=1}^{M+1}\prod_{j=1}^M|x_i-y_j|^{\delta-1}\prod_{1\le i<j\le M+1}(x_i-x_j)^{1-\delta}\prod_{1\le i<j \le M}(y_i-y_j)^{1-\delta}.
\end{align*}
Finally, we define the constant $\mathfrak{c}^{(\beta)}(k;q)$, which turns out to be the leading order coefficient in the asymptotics of $\textnormal{MoM}_N^{(\beta)}(k;q)$, by the following integral expression:
\begin{align}\label{ConstIntegralRep}
\mathfrak{c}^{(\beta)}(k;q)=\int_{\mathsf{X}\in \mathsf{I}_c(k;q)} \prod_{M=1}^{kq-1}\phi_{M,M+1}^{\left(\frac{2}{\beta}\right)}\left(x^{(M)},x^{(M+1)}\right)\prod_{M=0}^{kq-1}\phi_{M+1,M+2}^{\left(\frac{2}{\beta}\right)}\left(\tilde{x}^{(M+1)},\mathfrak{e}_c\left(\tilde{x}^{(M)}\right)\right) d\mathsf{X},
\end{align}
where we have used the convention $\tilde{x}^{(0)}=\emptyset$. It is important to note, as alluded to in the introduction and proven below, that when $k\ge 2$ and $\beta$ is large enough the integral defining $\mathfrak{c}^{(\beta)}(k;q)$ is infinite. Nevertheless, since the integrand is non-negative, it is unambiguously defined for all $\beta>0$. We also recall the definition of the set $\mathcal{A}(k;q)$ given by, for any fixed $k,q\in \mathbb{N}$:
\begin{align*}
 \mathcal{A}(k;q)=\big\{\beta>0: \mathfrak{c}^{(\beta)}(k;q)<\infty \big\}.
\end{align*}

We are now ready to prove Proposition \ref{PropAsymptotics1}.
\begin{proof}[Proof of Proposition \ref{PropAsymptotics1}] The idea is simple, namely a discrete to continuous scaling limit in going from a Riemann sum to an integral.

Using Proposition \ref{CombinatorialRepMoM} we can write:
\begin{align*}
 \textnormal{MoM}_N^{(\beta)}(k;q)&=N^{\frac{2}{\beta}(kq)^2-(k-1)}\frac{1}{N^{(kq)^2-(k-1)}}N^{-\left(\frac{2}{\beta}-1\right)(kq)^2}   \sum_{\mathsf{\Lambda}\in\mathsf{I}_N(k;q)}\psi_{\mathfrak{e}_N(\emptyset)/\tilde{\lambda}^{(1)}}^{\left(\frac{2}{\beta}\right)}\psi_{\mathfrak{e}_N(\tilde{\lambda}^{(1)})/\tilde{\lambda}^{(2)}}^{\left(\frac{2}{\beta}\right)}\times\cdots\\
 &\cdots \times \psi_{\mathfrak{e}_N(\tilde{\lambda}^{(kq-1)})/\lambda^{(kq)}}^{\left(\frac{2}{\beta}\right)}\psi_{\lambda^{(kq)}/\lambda^{(kq-1)}}^{\left(\frac{2}{\beta}\right)}\cdots  \psi_{\lambda^{(2)}/\lambda^{(1)}}^{\left(\frac{2}{\beta}\right)}\\
 &=N^{\frac{2}{\beta}(kq)^2-(k-1)}\frac{1}{N^{(kq)^2-(k-1)}}\sum_{\mathsf{\Lambda}\in\mathsf{I}_N(k;q)}N^{-\left(\frac{2}{\beta}-1\right)}\psi_{\mathfrak{e}_N(\emptyset)/\tilde{\lambda}^{(1)}}^{\left(\frac{2}{\beta}\right)}N^{-2\left(\frac{2}{\beta}-1\right)}\psi_{\mathfrak{e}_N(\tilde{\lambda}^{(1)})/\tilde{\lambda}^{(2)}}^{\left(\frac{2}{\beta}\right)}\times\cdots\\
 &\cdots \times N^{-kq\left(\frac{2}{\beta}-1\right)}\psi_{\mathfrak{e}_N(\tilde{\lambda}^{(kq-1)})/\lambda^{(kq)}}^{\left(\frac{2}{\beta}\right)}N^{-(kq-1)\left(\frac{2}{\beta}-1\right)}\psi_{\lambda^{(kq)}/\lambda^{(kq-1)}}^{\left(\frac{2}{\beta}\right)}\cdots  N^{-\left(\frac{2}{\beta}-1\right)}\psi_{\lambda^{(2)}/\lambda^{(1)}}^{\left(\frac{2}{\beta}\right)}.
\end{align*}
Now using $(t)_m=\frac{\Gamma(t+m)}{\Gamma(t)}$ we can rewrite the weight $\psi_{\lambda/\mu}^{(\delta)}$, for $\mu\in\mathsf{S}_+^{(M)}, \lambda \in \mathsf{S}_+^{(M+1)}$ that interlace, in the following suggestive way:
\begin{align*}
\psi_{\lambda/\mu}^{(\delta)}=\frac{1}{\Gamma(\delta)^M}  \prod_{1\le i<j\le M}\frac{\Gamma(\mu_i-\lambda_{j+1}+\delta(j-i)+\delta)\Gamma(\mu_i-\mu_j+\delta(j-i)+1)}{\Gamma(\mu_i-\lambda_{j+1}+\delta(j-i)+1)\Gamma(\mu_i-\mu_j+\delta(j-i)+\delta)}\\
\times \frac{\Gamma(\lambda_i-\lambda_{j+1}+\delta(j-i)+1)\Gamma(\lambda_i-\mu_j+\delta(j-i)+\delta)}{\Gamma(\lambda_i-\lambda_{j+1}+\delta(j-i)+\delta)\Gamma(\lambda_i-\mu_j+\delta(j-i)+1)}
\\\times \prod_{i=1}^M\frac{\Gamma(\mu_i-\lambda_{i+1}+\delta)\Gamma(\lambda_i-\lambda_{i+1}+1)\Gamma(\lambda_i-\mu_i+\delta)}{\Gamma(\mu_i-\lambda_{i+1}+1)\Gamma(\lambda_i-\lambda_{i+1}+\delta)\Gamma(\lambda_i-\mu_i+1)}.
\end{align*}
Then, making use\footnote{To be more precise, we can apply this approximation over the set $\mathsf{I}^\circ_N(k;q)$, defined as $\mathsf{I}_N(k;q)$ but with strict inequalities. Then it is easy to see that the contribution from the sum over $\mathsf{I}_N(k;q)\backslash\mathsf{I}^\circ_N(k;q)$ is of lower order in $N$ (we have at least one less free variable).} of the standard approximation:
\begin{align*}
\frac{\Gamma(z+c)}{\Gamma(z+d)}=z^{c-d}+\mathcal{O}_{c,d}\left(z^{c-d-1}\right),     
\end{align*}
we obtain the following:
\begin{align*}
 \textnormal{MoM}_N^{(\beta)}(k,q)\sim  N^{\frac{2}{\beta}(kq)^2-(k-1)}\frac{1}{N^{(kq)^2-(k-1)}}\sum_{\mathsf{\Lambda}\in\mathsf{I}_N(k;q)} \prod_{M=1}^{kq-1}\phi_{M,M+1}^{\left(\frac{2}{\beta}\right)}\left(\frac{\lambda^{(M)}}{N},\frac{\lambda^{(M+1)}}{N}\right)\\\prod_{M=0}^{kq-1}\phi_{M+1,M+2}^{\left(\frac{2}{\beta}\right)}\left(\frac{\tilde{\lambda}^{(M)}}{N},\frac{\mathfrak{e}_N\left(\tilde{\lambda}^{(M)}\right)}{N}\right).
\end{align*}
Then, using the Riemann sum approximation of an integral, which by assumption is finite since $\beta\in \mathcal{A}(k;q)$, we finally obtain:
\begin{align*}
  \textnormal{MoM}_N^{(\beta)}(k;q)&\sim  N^{\frac{2}{\beta}(kq)^2-(k-1)}\int_{\mathsf{X}\in \mathsf{I}_c(k;q)} \prod_{M=1}^{kq-1}\phi_{M,M+1}^{\left(\frac{2}{\beta}\right)}\left(x^{(M)},x^{(M+1)}\right)\prod_{M=0}^{kq-1}\phi_{M+1,M+2}^{\left(\frac{2}{\beta}\right)}\left(\tilde{x}^{(M+1)},\mathfrak{e}_c\left(\tilde{x}^{(M)}\right)\right) d\mathsf{X}\\
  &=N^{\frac{2}{\beta}(kq)^2-(k-1)}\mathfrak{c}^{(\beta)}(k;q).
\end{align*}
 To conclude we observe that $\mathfrak{c}^{(\beta)}(k;q)$ is strictly positive since $\mathsf{I}_c(k;q)$ has non-empty interior and the integrand is continuous and strictly positive when restricted there.
\end{proof}

\subsection{Proof of Proposition \ref{PropAsymptotics2}}\label{PropAsympt2Sec}

Before proving Proposition \ref{PropAsymptotics2} we briefly comment on why the integral defining $\mathfrak{c}^{(\beta)}(k;q)$ could be infinite when $k\ge 2$ and $\beta$ is large enough while it is always finite for $k=1$. Observe that while the integrands coming in the definition of $\mathfrak{c}^{(\beta)}(k;q)$ and $\mathfrak{c}^{(\beta)}(1;kq)$ are identical the corresponding integrals are over $\mathsf{I}_c(k;q)$ and $\mathsf{I}_c(1;kq)$ which are $(kq)^2-(k-1)$ and $(kq)^2$ dimensional respectively. Thus, it could be that for a fixed $\beta>0$ a singularity of the integrand of a certain order\footnote{Here we say that a function has a singularity of order $\mathfrak{o}$, where $\mathfrak{o}$ is a positive real number, at a point $\mathsf{x}_0$ if it blows up like $r^{-\mathfrak{o}}$ as $r\to 0$, where $r$ is the distance from $\mathsf{x}_0$. It is most convenient to think of a function in terms of spherical coordinates around $\mathsf{x}_0$, in which case and in the particular setting of this paper computing the order of the singularity at $\mathsf{x}_0$ boils down to a combinatorial power counting argument.} is integrable over $\mathsf{I}_c(1;kq)$, in fact this holds for all $\beta>0$, while it is not integrable over $\mathsf{I}_c(k;q)$ for $k\ge 2$ (this can only happen when $\beta$ is large enough since when $\beta\le 2$ the integrand is uniformly bounded as we show below).

We now prove a number of results which combined give Proposition \ref{PropAsymptotics2}. We begin with the following lemma on the dependence of the weight $\phi_{M,M+1}^{(\delta)}(y,x)$ on $\delta$.
\begin{lem}\label{WeightBoundedness}
Let $M\in \mathbb{N}$ and $\delta'\le \delta$. Then, for any $y\in \mathsf{W}_+^{(M)}\cap[0,1]^{M}, x\in \mathsf{W}_+^{(M+1)}\cap [0,1]^{M+1}$ which interlace we have:
\begin{align*}
    \Gamma(\delta)^M\phi_{M,M+1}^{(\delta)}(y,x)\le \Gamma(\delta')^M\phi_{M,M+1}^{(\delta')}(y,x).
\end{align*}
In particular, for $\delta \ge 1$ since $\phi^{(1)}_{M,M+1}(y,x)\equiv 1$ we have:
\begin{align*}
    \phi_{M,M+1}^{(\delta)}(y,x)\le \frac{1}{\Gamma(\delta)^M}.
\end{align*}
\end{lem}

\begin{proof} Observe that we can write:
\begin{align*}
   \Gamma(\delta)^M\phi_{M,M+1}^{(\delta)}(y,x)=\left[\phi_{M,M+1}(y,x)\right]^{\delta-1},  
\end{align*}
where $\phi_{M,M+1}$ is given by:
\begin{align*}
  \phi_{M,M+1}(y,x)&=  \prod_{1\le i<j \le M}(y_i-x_{j+1})(y_i-y_j)^{-1}(x_i-y_j)(x_i-x_{j+1})^{-1}\\&\times \prod_{i=1}^M(y_i-x_{i+1})(x_i-y_i)(x_i-x_{i+1})^{-1}.
\end{align*}
We show that, for any $y\in \mathsf{W}_+^{(M)}\cap[0,1]^{M}, x\in \mathsf{W}_+^{(M+1)}\cap [0,1]^{M+1}$ which interlace, $\phi_{M,M+1}(y,x)\le 1$ which suffices to establish the lemma.
We have that:
\begin{align*}
      \phi_{M,M+1}(y,x)&= \prod_{1\le i< j\le M}(y_i-x_{j+1})(y_i-y_j)^{-1} \left( \frac{x_i-y_j}{x_i-x_{j+1}}\right) \prod_{i=1}^M(y_i-x_{i+1})\left( \frac{x_i-y_i}{x_i-x_{i+1}}\right)\\ 
      &\le \prod_{1\le i< j\le M}(y_i-x_{j+1})(y_i-y_j)^{-1}\prod_{i=1}^M(y_i-x_{i+1}),
\end{align*}
since $x_{j+1}\le y_j$ because of the interlacing. We rewrite the last line as follows:
\begin{align*}
\underbrace{\prod_{1\le i<j\le M}(y_i-x_{j+1})\prod_{\substack{1\le i<j \le M \\ j\neq i+1}}(y_i-y_j)^{-1}}_\textrm{(I)}\underbrace{\prod_{i=1}^{M-1}(y_i-y_{i+1})^{-1} \prod_{i=1}^M(y_i-x_{i+1})}_\textrm{(II)}.
\end{align*}
The second term (II) can be bounded as follows:
\begin{align*}
 \textnormal{(II)}=\prod_{i=1}^{M-1}(y_i-y_{i+1})^{-1} \prod_{i=1}^M(y_i-x_{i+1})= (y_M-x_{M+1}) \prod_{i=1}^{M-1}\left(\frac{y_i-x_{i+1}}{y_i-y_{i+1}}\right) \le 1,
\end{align*}
since $y_{i+1}\le x_{i+1}$ and moreover $y_M\le 1$ and $x_{M+1}\ge 0$. While the first term (I) satisfies:
\begin{align*}
 \textnormal{(I)}&=\prod_{1\le i<j\le M}(y_i-x_{j+1})\prod_{\substack{1\le i<j \le M \\ j\neq i+1}}(y_i-y_j)^{-1} =\prod_{i=1}^M\left[\prod_{j=i+1}^M (y_i-x_{j+1})^{} \prod_{j=i+2}^M(y_i-y_j)^{-1} \right]\\
 &=\prod_{i=1}^M\left[\left(\frac{y_i-x_{i+2}}{y_i-y_{i+2}}\right) \cdots \left( \frac{y_i-x_M}{y_i-y_M}\right) (y_i-x_{M+1})\right] \le 1,
\end{align*}
again since $y_j\le x_j$ and also $y_i\le 1$ and $x_{M+1}\ge 0$. The conclusion follows.
\end{proof} 

The lemma above gives us the following corollary on the dependence on $\beta$ of the leading order coefficient in the asymptotics.

\begin{cor}\label{CorComparison}
Let $k,q\in \mathbb{N}$ and $\beta \le \beta'$. Then, we have:
\begin{align*}
\mathfrak{c}^{(\beta)}(k;q)\le   \left[\frac{\Gamma\left(\frac{2}{\beta'}\right)}{\Gamma\left(\frac{2}{\beta}\right)}\right]^{(kq)^2}\mathfrak{c}^{(\beta')}(k;q).
\end{align*}
In particular, if $\beta \le 2$ then:
\begin{align*}
\mathfrak{c}^{(\beta)}(k;q)\le \frac{1}{\Gamma\left(\frac{2}{\beta}\right)^{(kq)^2}}\mathfrak{c}^{(2)}(k;q)=\frac{1}{\Gamma\left(\frac{2}{\beta}\right)^{(kq)^2}} \textnormal{volume}\left(\mathsf{I}_c(k;q)\right)<\infty.
\end{align*}
\end{cor}
\begin{rmk}
Observe that for $k=1$ using the explicit formula (\ref{k=1explicit}) for $\mathfrak{c}^{(\beta)}(1;q)$, see also Lemma \ref{ExplicitExpressions}, the corollary above is equivalent to:
\begin{align}\label{Non-Increasing}
t\mapsto \Gamma(t)^{q^2}\prod_{j=1}^q\frac{\Gamma\left(tj\right)}{\Gamma\left(t\left(q+j\right)\right)} \ \ \textnormal{ is non-increasing on } (0,\infty).
\end{align}
This elementary statement can alternatively be proven directly as follows. It suffices to show that the logarithm of (\ref{Non-Increasing}) is non-increasing on $(0,\infty)$. The logarithmic derivative of (\ref{Non-Increasing}) is equal to, using the notation $\mathsf{D}(x)=\frac{d}{dx}\log\Gamma(x)$:
\begin{align*}
  q^2\mathsf{D}(t)+\sum_{j=1}^q \left[j\mathsf{D}\left(jt\right)-(q+j)\mathsf{D}\left(t(q+j)\right)\right] &\le   q^2\mathsf{D}(t)+\sum_{j=1}^q \left[j\mathsf{D}\left(jt\right)-(q+j)\mathsf{D}\left(jt\right)\right]\\
  &=q^2\mathsf{D}(t)-q\sum_{j=1}^q\mathsf{D}(jt)\le q^2\mathsf{D}(t)-q^2\mathsf{D}(t)=0,
\end{align*}
where we have used the well-known fact that $\mathsf{D}(\cdot)$ is non-decreasing on $(0,\infty)$. Then, (\ref{Non-Increasing}) follows.
\end{rmk}
In the special cases $k=1$ and $k=2$ it is possible to give simpler integral expressions for $\mathfrak{c}^{(\beta)}(1;q)$ and $\mathfrak{c}^{(\beta)}(2;q)$ by computing the intermediate integrals over the interlacing arrays, except for the row at which the arrays are joined at. In the case of $\mathfrak{c}^{(\beta)}(1;q)$ the resulting integral, which is a special case of the Selberg integral, see for example \cite{Forrester}, can be computed explicitly and gives (\ref{k=1explicit}).

\begin{lem}\label{ExplicitExpressions} Let $q\in \mathbb{N}$ and $\beta>0$. Then, we have the following expressions:
\begin{align}
   \mathfrak{c}^{(\beta)}(1;q) &=\prod_{M=1}^{q}\frac{\Gamma\left(\frac{2}{\beta}\right)}{\Gamma\left(M\frac{2}{\beta}\right)^2}\int_{\mathsf{W}_+^{(q)}\cap[0,1]^{q}}\prod_{1\le i<j \le q}(x_i-x_j)^{\frac{4}{\beta}}\prod_{i=1}^{q}\left[x_i\left(1-x_i\right)\right]^{\frac{2}{\beta}-1}dx\\&=\prod_{i=1}^q\frac{\Gamma\left(\frac{2}{\beta}i\right)}{\Gamma\left(\frac{2}{\beta}(q+i)\right)}\nonumber,\\
    \mathfrak{c}^{(\beta)}(2;q)&=\prod_{M=1}^{2q}\frac{\Gamma\left(\frac{2}{\beta}\right)}{\Gamma\left(M\frac{2}{\beta}\right)^2}\int_{\mathsf{W}_+^{(2q)}\cap[0,1]^{2q},\  \sum_{i=1}^{2q}x_i=q}\prod_{1\le i<j \le 2q}(x_i-x_j)^{\frac{4}{\beta}}\prod_{i=1}^{2q}\left[x_i\left(1-x_i\right)\right]^{\frac{2}{\beta}-1}dx.\label{Integralk=2}
\end{align}
\end{lem}

\begin{proof}
Observe that we have:
\begin{align}
 \prod_{M=1}^{kq-1}\phi_{M,M+1}^{(\delta)}\left(x^{(M)},x^{(M+1)}\right)=\prod_{M=1}^{kq-1}\frac{1}{\Gamma(\delta)^M}\prod_{M=1}^{kq-1}\prod_{i=1}^{M+1}\prod_{j=1}^M\big|x_i^{(M+1)}-x_j^{(M)}\big|^{\delta-1}\nonumber\\\times\prod_{M=2}^{kq-1}\prod_{1\le i<j\le M}\left(x_i^{(M)}-x_j^{(M)}\right)^{2-2\delta}
 \prod_{1\le i<j\le kq}\left(x_i^{(kq)}-x_j^{(kq)}\right)^{1-\delta}\label{weightexplicitdisplay1}
\end{align}
and similarly (recall that $x^{(kq)}=\tilde{x}^{(kq)}$):
\begin{align}
     \prod_{M=0}^{kq-1}\phi_{M+1,M+2}^{(\delta)}\left(\tilde{x}^{(M+1)},\mathfrak{e}_c\left(\tilde{x}^{(M)}\right)\right)=\prod_{M=1}^{kq}\frac{1}{\Gamma(\delta)^M}\prod_{M=1}^{kq-1}\prod_{i=1}^{M+1}\prod_{j=1}^M\big|\tilde{x}_i^{(M+1)}-\tilde{x}_j^{(M)}\big|^{\delta-1}\nonumber\\
     \times\prod_{M=2}^{kq-1}\prod_{1\le i<j\le M}\left(\tilde{x}_i^{(M)}-\tilde{x}_j^{(M)}\right)^{2-2\delta} \prod_{1\le i<j\le kq}\left(x_i^{(kq)}-x_j^{(kq)}\right)^{1-\delta}\prod_{i=1}^{kq}\left[x_i^{(kq)}\left(1-x_i^{(kq)}\right)\right]^{\delta-1}.\label{weightexplicitdisplay2}
\end{align}
We notice that these weights are up to a factor involving only $x^{(kq)}$, given by the orbital beta probability distribution on continuous interlacing arrays with fixed top row $x^{(kq)}$, see Definition 1.3 in \cite{GorinMarcus}, also \cite{CuencaOrbital}, \cite{AssiotisNajnudel}. 

We now argue as follows. For $k=1$ and $k=2$ we fix the centre row $x^{(kq)}$ (note that for $k=2$ there is a single sum constraint only on this row) and perform the integrations over the two individual arrays with fixed top row $x^{(kq)}$. This choice of the order of integration is possible by Tonelli's theorem since the integrand is positive.  Then, the corresponding integrals over the interlacing arrays with fixed top row $x^{(kq)}$ are known to have an explicit evaluation given as follows:
\begin{align*}
 \int_{x^{(1)}\prec x^{(2)}\prec x^{(3)}\prec \dots \prec x^{(kq-1)}\prec x^{(kq)}}  \prod_{M=1}^{kq-1}\phi_{M,M+1}^{(\delta)}\left(x^{(M)},x^{(M+1)}\right) dx^{(1)}dx^{(2)}dx^{(3)}\dots dx^{(kq-1)}\\
 =\Gamma(\delta)^{kq}\prod_{M=1}^{kq}\frac{1}{\Gamma(M\delta)}\prod_{1\le i<j\le kq}\left(x_i^{(kq)}-x_j^{(kq)}\right)^\delta,\\
  \int_{\tilde{x}^{(1)}\prec \tilde{x}^{(2)}\prec \tilde{x}^{(3)}\prec \dots \prec \tilde{x}^{(kq-1)}\prec \tilde{x}^{(kq)}=x^{(kq)}}  \prod_{M=0}^{kq-1}\phi_{M,M+1}^{(\delta)}\left(\tilde{x}^{(M+1)},\mathfrak{e}_c\left(\tilde{x}^{(M)}\right)\right) d\tilde{x}^{(1)}d\tilde{x}^{(2)}d\tilde{x}^{(3)}\dots d\tilde{x}^{(kq-1)}\\
  =\prod_{M=1}^{kq}\frac{1}{\Gamma(M\delta)}\prod_{1\le i<j\le kq}\left(x_i^{(kq)}-x_j^{(kq)}\right)^\delta\prod_{i=1}^{kq}\left[x_i^{(kq)}\left(1-x_i^{(kq)}\right)\right]^{\delta-1}.
\end{align*}
This is exactly the evaluation of the normalisation constant for the orbital beta probability distribution given in displays (8) and (9) of Definition 1.3 in \cite{GorinMarcus}, see also Remark 1.4 therein. We note that the argument above would not work for $k\ge 3$ as the sum constraints are involved not only on the row where the arrays are joined.

This readily gives the integral expressions for $\mathfrak{c}^{(\beta)}(1;q)$ and $\mathfrak{c}^{(\beta)}(2;q)$. The final expression for $\mathfrak{c}^{(\beta)}(1;q)$ is then an immediate consequence of the explicit evaluation of Selberg's integral, see \cite{Forrester}. As far as we can tell the integral expression for $\mathfrak{c}^{(\beta)}(2;q)$ is not known to have an explicit evaluation and we leave it in the form (\ref{Integralk=2}).
\end{proof}

Using formula (\ref{Integralk=2}) it is relatively straightforward to show that $\mathcal{A}(2;q)=(0,4q^2)$.

\begin{lem}\label{LemmaIntegrabilityk=2}
Let $q \in \mathbb{N}$. Then, $\mathcal{A}(2;q)=(0,4q^2)$.
\end{lem}

\begin{proof}
We show that the integral (\ref{Integralk=2}) is finite if and only if $\beta<4q^2$ (recall that by definition $\beta>0$). Observe that the integral is $(2q-1)$-dimensional, by eliminating one of the variables due to the sum constraint, for example $x_{2q}=q-\sum_{i=1}^{2q-1}x_i$. We first claim and we will justify at the end of the proof that ``singular manifolds" which are not full-dimensional, i.e. singularities over a number of variables $\mathsf{n}<2q-1$, are always integrable for any $\beta>0$. Hence, for now we restrict attention to the full-dimensional case. Observe that the only singularities occur when the coordinates are either $0$ or $1$ and each such singularity is of order\footnote{Clearly when $\beta \le 2$ there are no singularities which is consistent with the fact that the integrand is uniformly bounded when $\beta \le 2$ as proven in Lemma \ref{WeightBoundedness}.} $\left(1-\frac{2}{\beta}\right)$. The most singular points are then the ones all of whose coordinates are either $0$ or $1$ which by the sum constraint and ordering of the coordinates uniquely identifies a single point $\mathfrak{x}_{\star}(q)$ given by\footnote{We have written out $\mathfrak{x}_{\star}(q)$ as a $2q$-dimensional point but recall that one of the coordinates is completely determined by the sum constraint.}:
\begin{align*}
\mathfrak{x}_{\star}(q)=(\underbrace{1,\dots,1}_\textrm{$q$},\underbrace{0,\dots,0}_\textrm{$q$}).
\end{align*}
If any of the coordinates of a point are not $0$ or $1$ then the singularity at that point is automatically integrable for any $\beta>0$. Finally, we need to take into account that for each pair of coordinates which coalesce there is a factor vanishing at order $\frac{4}{\beta}$ in the integrand. Thus, putting everything together we obtain that the singularity at $\mathfrak{x}_{\star}(q)$ is of order:
\begin{equation*}
 2q\left(1-\frac{2}{\beta}\right)+\left[\binom{q}{2}+\binom{q}{2}
\right]\left(-\frac{4}{\beta}\right)=2q-\frac{4}{\beta}q^2. 
\end{equation*}
 This is integrable if and only if\footnote{In this case, a singularity of order $\mathfrak{o}$ is integrable if and only if $\mathfrak{o}<\mathfrak{d}$ where $\mathfrak{d}$ is the dimension of the integral at hand; making use of spherical coordinates around the singular point, since we get a factor $r^{\mathfrak{d}-1}$ from the Jacobian, this boils down to the integrability of $r^{-\mathfrak{o}+\mathfrak{d}-1}$ at $0$.}:
 \begin{equation*}
2q-\frac{4}{\beta}q^2 < 2q-1,
 \end{equation*}
 which gives $\beta<4q^2$.
 
 Returning to the claim we made earlier, suppose that we are looking at a singularity over a number of variables $\mathsf{n}<2q-1$. Suppose $\mathsf{n}_0$ of these variables are $0$ and $\mathsf{n}_1$ are $1$, so that $\mathsf{n}_0+\mathsf{n}_1\le \mathsf{n}$ (we must also have $\mathsf{n}_1\le q$ due to the sum constraint but we will not need to use this extra restriction). Then, by completely analogous considerations to the ones above for $\mathfrak{x}_\star(q)$, the order of the singularity at such a point is given by (note that when $\mathsf{n}=2q-1$ and in the particular case of $\mathfrak{x}_\star(q)$ above we also picked up the singularity coming from the fixed coordinate $x_{2q}$):
 \begin{equation*}
\left(\mathsf{n}_0+\mathsf{n}_1\right)\left(1-\frac{2}{\beta}\right)+ \left[\binom{\mathsf{n}_0}{2}+\binom{\mathsf{n}_1}{2}\right]\left(-\frac{4}{\beta}\right)   =   \mathsf{n}_0+\mathsf{n}_1-\left(\mathsf{n}_0^2+\mathsf{n}_1^2\right)\frac{2}{\beta}.
 \end{equation*}
 This singularity is then integrable for any $\beta>0$, since:
 \begin{equation*}
   \mathsf{n}_0+\mathsf{n}_1-\left(\mathsf{n}_0^2+\mathsf{n}_1^2\right)\frac{2}{\beta}  <\mathsf{n}
 \end{equation*}
 and this concludes the proof.
 \end{proof}
 
 \begin{figure}[hbt!]

\scalebox{0.66}{
\begin{tikzpicture}
 \draw [fill=gray!20] (4,4) -- (2,6.5)-- (0,4)--(2,1.5); 
 
  \draw [fill=gray!20] (4,4) -- (6,6.5)-- (8,4)--(6,1.5); 

\draw[ultra thick, blue]   (4,-1) -- (0,4);
\draw[ultra thick,blue]   (4,9) -- (0,4);
\draw[ultra thick,blue]   (4,9) -- (8,4);
 \draw[ultra thick,blue]   (4,-1) -- (8,4);
 
  \node[] at (2,5) {\LARGE \bf{0}};

     \node[] at (6,5) {\LARGE \bf{1}};

   \draw[ultra thick,red]   (0,4) -- (8,4);

     \draw[ultra thick]   (4,4) -- (2,6.5);      
    \draw[ultra thick]   (4,9) -- (2,6.5);      
     \draw[ultra thick]   (4,4) -- (6,6.5);      
     \draw[ultra thick]   (4,9) -- (6,6.5);

        \draw[ultra thick]   (4,4) -- (2,1.5);      
        \draw[ultra thick]   (4,4) -- (6,1.5);      
        \draw[ultra thick]   (4,-1) -- (6,1.5);      
        \draw[ultra thick]   (4,-1) -- (2,1.5);

\draw[fill,green] (4,-0.8) circle [radius=0.1];   
\draw[fill,green] (3.7,-0.4) circle [radius=0.1];   
\draw[fill,green] (4.3,-0.4) circle [radius=0.1];   
\draw[fill,green] (3.4,0) circle [radius=0.1];   
\draw[fill,green] (4,0) circle [radius=0.1];   
\draw[fill,green] (4.6,0) circle [radius=0.1];   

  \node[] at (4,0.6) {$\boldsymbol{\vdots}$};

\draw[fill,green] (4,3.75) circle [radius=0.1];   
\draw[fill,green] (3.7,3.4) circle [radius=0.1];   
\draw[fill,green] (4.3,3.4) circle [radius=0.1];   
\draw[fill,green] (3.4,3) circle [radius=0.1];   
\draw[fill,green] (4,3) circle [radius=0.1];   
\draw[fill,green] (4.6,3) circle [radius=0.1];  

  \node[] at (4,2.4) {$\boldsymbol{\vdots}$};

 \draw[fill,green] (4,8.8) circle [radius=0.1];   
\draw[fill,green] (3.7,8.4) circle [radius=0.1];   
\draw[fill,green] (4.3,8.4) circle [radius=0.1];   
\draw[fill,green] (3.4,8) circle [radius=0.1];   
\draw[fill,green] (4,8) circle [radius=0.1];   
\draw[fill,green] (4.6,8) circle [radius=0.1];   

  \node[] at (4,7.4) {$\boldsymbol{\vdots}$};

\draw[fill,green] (4,4.25) circle [radius=0.1];   
\draw[fill,green] (3.7,4.6) circle [radius=0.1];   
\draw[fill,green] (4.3,4.6) circle [radius=0.1];   
\draw[fill,green] (3.4,5) circle [radius=0.1];   
\draw[fill,green] (4,5) circle [radius=0.1];   
\draw[fill,green] (4.6,5) circle [radius=0.1];  

  \node[] at (4,5.6) {$\boldsymbol{\vdots}$};

  \node[] at (6,8) {\Large \bf{k=2}};

\end{tikzpicture}

\begin{tikzpicture}

 \draw [fill=gray!20] (3.2,0) -- (4,1)--(3.2,2)--(4,3)--(3.2,4)--(4,5)--(3.2,6)--(4,7)--(3.2,8) --(0,4)--(3.2,0); 
 
  \draw [fill=gray!20] (4.8,0) -- (4,1)--(4.8,2)--(4,3)--(4.8,4)--(4,5)--(4.8,6)--(4,7)--(4.8,8) --(8,4)--(4.8,0);

\draw[ultra thick, blue]   (4,-1) -- (0,4);
\draw[ultra thick,blue]   (4,9) -- (0,4);
\draw[ultra thick,blue]   (4,9) -- (8,4);
 \draw[ultra thick,blue]   (4,-1) -- (8,4);

   \draw[ultra thick,red]   (2.4,7) -- (5.6,7);

       \draw[ultra thick,red]   (2.4,1) -- (5.6,1);         
      \draw[ultra thick,red]   (0.8,3) -- (7.2,3);            
      
\draw[ultra thick,red]   (0.8,5) -- (7.2,5);

  \draw[ultra thick]   (4,7) -- (3.2,8);      
    \draw[ultra thick]   (3.2,8) -- (4,9);      
     \draw[ultra thick]   (4,9) -- (4.8,8);      
     \draw[ultra thick]   (4.8,8) -- (4,7);      
     
   \draw[ultra thick]   (4,-1) -- (3.2,0);      
    \draw[ultra thick]   (3.2,0) -- (4,1);      
     \draw[ultra thick]   (4.8,0) -- (4,1);      
     \draw[ultra thick]   (4.8,0) -- (4,-1);     
     
   \draw[ultra thick]   (4,1) -- (3.2,2);      
    \draw[ultra thick]   (3.2,2) -- (4,3);      
     \draw[ultra thick]   (4.8,2) -- (4,3);      
     \draw[ultra thick]   (4.8,2) -- (4,1);     
     
     \draw[ultra thick]   (4,3) -- (3.2,4);      
    \draw[ultra thick]   (3.2,4) -- (4,5);      
     \draw[ultra thick]   (4,5) -- (4.8,4);      
     \draw[ultra thick]   (4.8,4) -- (4,3);

     \draw[ultra thick]   (4,5) -- (3.2,6);      
    \draw[ultra thick]   (3.2,6) -- (4,7);      
     \draw[ultra thick]   (4,7) -- (4.8,6);      
     \draw[ultra thick]   (4.8,6) -- (4,5);

  \node[] at (2,4) {\LARGE \bf{0}};

     \node[] at (6,4) {\LARGE \bf{1}};

 \draw[fill,green] (4,-0.75) circle [radius=0.1];   
\draw[fill,green] (3.8,-0.45) circle [radius=0.1];   
\draw[fill,green] (4.2,-0.45) circle [radius=0.1];

  \node[] at (4,0.1) {$\boldsymbol{\vdots}$}; 
  
  \draw[fill,green] (4,0.75) circle [radius=0.1];   
\draw[fill,green] (3.8,0.45) circle [radius=0.1];   
\draw[fill,green] (4.2,0.45) circle [radius=0.1];

 \draw[fill,green] (4,1.25) circle [radius=0.1];   
\draw[fill,green] (3.8,1.55) circle [radius=0.1];   
\draw[fill,green] (4.2,1.55) circle [radius=0.1];

  \node[] at (4,2.1) {$\boldsymbol{\vdots}$}; 
  
  \draw[fill,green] (4,2.75) circle [radius=0.1];   
\draw[fill,green] (3.8,2.45) circle [radius=0.1];   
\draw[fill,green] (4.2,2.45) circle [radius=0.1];

 \draw[fill,green] (4,3.25) circle [radius=0.1];   
\draw[fill,green] (3.8,3.55) circle [radius=0.1];   
\draw[fill,green] (4.2,3.55) circle [radius=0.1];

  \node[] at (4,4.1) {$\boldsymbol{\vdots}$}; 
  
  \draw[fill,green] (4,4.75) circle [radius=0.1];   
\draw[fill,green] (3.8,4.45) circle [radius=0.1];   
\draw[fill,green] (4.2,4.45) circle [radius=0.1];

 \draw[fill,green] (4,5.25) circle [radius=0.1];   
\draw[fill,green] (3.8,5.55) circle [radius=0.1];   
\draw[fill,green] (4.2,5.55) circle [radius=0.1];

  \node[] at (4,6.1) {$\boldsymbol{\vdots}$}; 
  
  \draw[fill,green] (4,6.75) circle [radius=0.1];   
\draw[fill,green] (3.8,6.45) circle [radius=0.1];   
\draw[fill,green] (4.2,6.45) circle [radius=0.1];

 \draw[fill,green] (4,7.25) circle [radius=0.1];   
\draw[fill,green] (3.8,7.55) circle [radius=0.1];   
\draw[fill,green] (4.2,7.55) circle [radius=0.1];

  \node[] at (4,8.1) {$\boldsymbol{\vdots}$}; 
  
  \draw[fill,green] (4,8.75) circle [radius=0.1];   
\draw[fill,green] (3.8,8.45) circle [radius=0.1];   
\draw[fill,green] (4.2,8.45) circle [radius=0.1];

  \node[] at (6,8) {\Large \bf{k=5}};

\end{tikzpicture}

\begin{tikzpicture}

 \draw [fill=gray!20] (0.8,0.8) -- (2,2)--(0.8,3.2)--(0.4,2.8)--(1.2,2)--(0.4,1.2)--(0.8,0.8); 
 \draw [fill=gray!20] (-0.8,0.8) -- (-2,2)--(-0.8,3.2)--(-0.4,2.8)--(-1.2,2)--(-0.4,1.2)--(-0.8,0.8); 

  \node[] at (1.6,2) {\Large \bf{1}};
  \node[] at (-1.6,2) {\Large \bf{0}};

  \node[] at (1.2,4.3) {\Large \bf{k=3,q=2}};
  
     \draw[ultra thick,red]   (-1.2,1.2) -- (1.2,1.2);      
     \draw[ultra thick,red]   (-1.2,2.8) -- (1.2,2.8);

  \draw[fill,green] (0,0) circle [radius=0.1];   
    \draw[fill,green] (-0.4,0.4) circle [radius=0.1];   
       \draw[fill,green] (0.4,0.4) circle [radius=0.1];   
     \draw[fill,green] (0,0.8) circle [radius=0.1];

    \draw[fill] (-0.8,0.8) circle [radius=0.1];   
        \draw[fill] (0.8,0.8) circle [radius=0.1];   
        \draw[fill] (-1.2,1.2) circle [radius=0.1];   
        \draw[fill] (1.2,1.2) circle [radius=0.1];   
        \draw[fill] (-0.4,1.2) circle [radius=0.1];   
        \draw[fill] (0.4,1.2) circle [radius=0.1];   
        
     \draw[fill,green] (0,1.6) circle [radius=0.1];   
    \draw[fill,green] (-0.4,2) circle [radius=0.1];   
       \draw[fill,green] (0.4,2) circle [radius=0.1];   
     \draw[fill,green] (0,2.4) circle [radius=0.1];

           \draw[fill,green] (0,3.2) circle [radius=0.1];   
    \draw[fill,green] (-0.4,3.6) circle [radius=0.1];   
       \draw[fill,green] (0.4,3.6) circle [radius=0.1];   
     \draw[fill,green] (0,4) circle [radius=0.1];

         \draw[fill] (-0.8,3.2) circle [radius=0.1];   
        \draw[fill] (0.8,3.2) circle [radius=0.1];   
        \draw[fill] (-1.2,2.8) circle [radius=0.1];   
        \draw[fill] (1.2,2.8) circle [radius=0.1];   
        \draw[fill] (-0.4,2.8) circle [radius=0.1];   
        \draw[fill] (0.4,2.8) circle [radius=0.1]; 
        
              \draw[fill] (0.8,2.4) circle [radius=0.1];   
        \draw[fill] (-0.8,2.4) circle [radius=0.1];   
        \draw[fill] (1.6,2.4) circle [radius=0.1];   
        \draw[fill] (-1.6,2.4) circle [radius=0.1];   

                 \draw[fill] (0.8,1.6) circle [radius=0.1];
        \draw[fill] (-0.8,1.6) circle [radius=0.1];   
        \draw[fill] (1.6,1.6) circle [radius=0.1];   
        \draw[fill] (-1.6,1.6) circle [radius=0.1];   
        
      \draw[fill] (1.2,2) circle [radius=0.1];
        \draw[fill] (-1.2,2) circle [radius=0.1];   
        \draw[fill] (2,2) circle [radius=0.1];   
        \draw[fill] (-2,2) circle [radius=0.1];

\end{tikzpicture}

}

\caption{The figures depict the restricted set of variables that we are integrating over in the proof of Lemma \ref{LemmaSupremum} shown shaded in grey. The variables we have removed (not integrating over) are depicted as green particles. As shown in the figure these correspond to the $k$ squares defined in the proof of Lemma \ref{LemmaSupremum}. The rows with the sum constraints are depicted as solid red lines. The figures also depict the definition of the point $\mathfrak{x}(k;q)$ half of whose coordinates are $0$'s and the other half $1$'s as shown here. Finally, observe that in the special case $k=2$ the point $\mathfrak{x}(2;q)$ corresponds to the point $\mathfrak{x}_{\star}(q)$ from the proof of Lemma \ref{LemmaIntegrabilityk=2}.}\label{Figure2}
\end{figure}
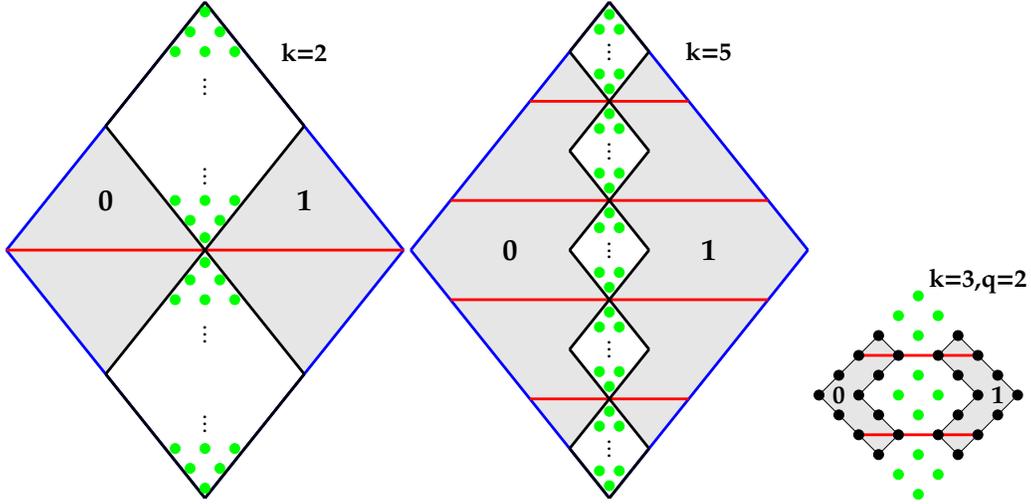

We now move on to prove the following result on finiteness of $\mathfrak{c}^{(\beta)}(k;q)$ when $k\ge 2$; clearly for $k=2$ this is a consequence of Lemma \ref{LemmaIntegrabilityk=2}. Unfortunately for $k\ge 3$, as far as we are aware, there is no analogous simplification as in Lemma \ref{ExplicitExpressions} and we need to analyse the integral over the whole interlacing array with constraints. Then, to prove this result we simply exhibit a singularity of the integrand which is not integrable if $\beta\ge 2kq^2$. This choice of singularity might seem to come out of thin air but we give some intuition for it after the proof. We expect that this singularity is in fact the optimal one in the sense that it gives the strictest restriction on $\beta$, which would show that $\mathcal{A}(k;q)=(0,2kq^2)$. We give some brief heuristics in support of this claim after the proof of the result.

Finally, the reader is advised to study Figure \ref{Figure2} and Figure \ref{Figure3}, while reading the proof of Lemma \ref{LemmaSupremum}, which help elucidate the argument.

\begin{lem}\label{LemmaSupremum} 
Let $k,q \in \mathbb{N}$ with $k\ge 2$. Then, $[2kq^2,\infty)\cap\mathcal{A}(k;q)=\emptyset$.

\end{lem}

\begin{proof} We will show that the integral over a restricted set of variables, that we define next, is infinite when $\beta\ge 2kq^2$. We remove (from the set of all variables), namely we do not integrate over, the variables corresponding to $k$ squares of coordinates, as shown in Figure \ref{Figure2}, each of side $q$. The $i$-th square is uniquely determined by two of its diagonal vertices which are given by the mid-coordinate of the first row above (the row corresponding to) the $(i-1)$-th sum constraint and the mid-coordinate of the last row below (the row corresponding to) the $i$-th sum constraint, see Figure \ref{Figure2}. For the extreme cases $i=1$ and $i=k$ we take as the lower and upper vertices of the corresponding squares the coordinates $x_{1}^{(1)}$ and $\tilde{x}_1^{(1)}$ respectively, see Figure \ref{Figure2}. We have refrained from giving here the definition in terms of the indices of coordinates since it is too cumbersome and somewhat obscures the simple geometric picture. Observe that the corresponding integral over the restricted set of variables just defined is $(kq)^2-kq^2-(k-1)=(kq^2-1)(k-1)$ dimensional, since we have removed $kq^2$ coordinates and we include all of the $(k-1)$ sum constraints which fix a variable each.

Restricting to the set of variables just described we then consider the point $\mathfrak{x}(k;q)$ whose coordinates consist entirely of $0$'s and $1$'s, see Figure \ref{Figure2} for an illustration. Note that this description uniquely identifies $\mathfrak{x}(k;q)$: the rows involving the sum constraints are uniquely determined (half of the coordinates are $0$'s and the other half $1$'s and they are also ordered) and then the rest of the coordinates of $\mathfrak{x}(k;q)$ are determined by the interlacing, see Figure \ref{Figure2}.

Moreover, a direct but tedious combinatorial computation using formulae (\ref{weightexplicitdisplay1}), (\ref{weightexplicitdisplay2}) (see Figure \ref{Figure3} for some explanations and the computation in simple cases, the general case is analogous but notationally cumbersome) gives that the singularity at the point $\mathfrak{x}(k;q)$ is of order\footnote{As before this is non-positive for $\beta\le 2$ which is consistent with the fact that the integrand is uniformly bounded when $\beta \le 2$ as proven in Lemma \ref{WeightBoundedness}.} $kq^2(k-1)\left(1-\frac{2}{\beta}\right)$.
This singularity is then not integrable if:
\begin{align*}
kq^2(k-1)\left(1-\frac{2}{\beta}\right)\ge(kq^2-1)(k-1),    
\end{align*}
which gives that for $\beta\ge2kq^2$ the integral is infinite.
\end{proof}

\begin{figure}[hbt!]
\centering
\begin{tikzpicture}

  \node[] at (1.2,2.5) { \bf{k=3,q=1}};

     \draw[red,dashed]   (-0.5,0.5) -- (0.5,0.5);      
     \draw[red,dashed]   (-0.5,1.5) -- (0.5,1.5);      

     \draw[ultra thick]   (-0.5,0.5) -- (-1,1);      
     \draw[ultra thick]   (-0.5,1.5) -- (-1,1);   
     
       \draw[ultra thick]   (0.5,0.5) -- (1,1);      
     \draw[ultra thick]   (0.5,1.5) -- (1,1);  
     
     \draw[fill,green] (0,0) circle [radius=0.1];   
    \draw[fill] (-0.5,0.5) circle [radius=0.1];   
       \draw[fill] (0.5,0.5) circle [radius=0.1];   
     \draw[fill,green] (0,1) circle [radius=0.1];   
          \draw[fill,blue] (-1,1) circle [radius=0.1];   
     \draw[fill,blue] (1,1) circle [radius=0.1];   
          \draw[fill,green] (0,2) circle [radius=0.1];   
    \draw[fill] (-0.5,1.5) circle [radius=0.1];   
       \draw[fill] (0.5,1.5) circle [radius=0.1];  
       
       \node[above right] at (1,1){\bf{1}};
       \node[above right] at (0.5,1.5){\bf{1}};
       \node[below right] at (0.5,0.5){\bf{1}};

       \node[above left] at (-1,1){\bf{0}};
       \node[above left] at (-0.5,1.5){\bf{0}};
       \node[below left] at (-0.5,0.5){\bf{0}};
       
\end{tikzpicture}
\begin{tikzpicture}

  \node[] at (1.2,3.5) { \bf{k=4,q=1}};

     \draw[red,dashed]   (-0.5,0.5) -- (0.5,0.5);      
     \draw[red,dashed]   (-1.5,1.5) -- (1.5,1.5);      
     \draw[red,dashed]   (-0.5,2.5) -- (0.5,2.5);   
      \draw[ultra thick,dashed] (-1.5,1.5) to [out=45,in=135] (-0.5,1.5);
      \draw[ultra thick,dashed] (0.5,1.5) to [out=45,in=135] (1.5,1.5);

  \draw[ultra thick]   (-0.5,0.5) -- (-1,1);      
     \draw[ultra thick]   (0.5,0.5) -- (1,1);      
     \draw[ultra thick]   (-1,1) -- (-1.5,1.5);      
     \draw[ultra thick]   (-1,1) -- (-0.5,1.5);      
     \draw[ultra thick]   (1,1) -- (1.5,1.5);      
     \draw[ultra thick]   (1,1) -- (0.5,1.5);
          \draw[ultra thick]   (-1,2) -- (-1.5,1.5);      
     \draw[ultra thick]   (-1,2) -- (-0.5,1.5);      
     \draw[ultra thick]   (1,2) -- (1.5,1.5);      
     \draw[ultra thick]   (1,2) -- (0.5,1.5);
       \draw[ultra thick]   (-1,2) -- (-0.5,2.5);      
     \draw[ultra thick]   (1,2) -- (0.5,2.5);

   \draw[fill,green] (0,0) circle [radius=0.1];   
    \draw[fill] (-0.5,0.5) circle [radius=0.1];   
       \draw[fill] (0.5,0.5) circle [radius=0.1];   
     \draw[fill,green] (0,1) circle [radius=0.1];   
          \draw[fill] (-1,1) circle [radius=0.1];   
     \draw[fill] (1,1) circle [radius=0.1];   
        \draw[fill,blue] (-1.5,1.5) circle [radius=0.1];   
          \draw[fill,blue] (-0.5,1.5) circle [radius=0.1];   
     \draw[fill,blue] (0.5,1.5) circle [radius=0.1];  
          \draw[fill,blue] (1.5,1.5) circle [radius=0.1];  
      \draw[fill,green] (0,2) circle [radius=0.1];   
   \draw[fill,green] (0,3) circle [radius=0.1];   
          \draw[fill] (-1,2) circle [radius=0.1];   
     \draw[fill] (1,2) circle [radius=0.1];   
        \draw[fill] (-0.5,2.5) circle [radius=0.1];   
     \draw[fill] (0.5,2.5) circle [radius=0.1];

           \node[above right] at (1.5,1.5){\bf{1}};
       \node[above left] at (0.5,1.5){\bf{1}};
       \node[below right] at (0.5,0.5){\bf{1}};
       \node[below right] at (1,1){\bf{1}};
        \node[above right] at (1,2){\bf{1}};
        \node[above right] at (0.5,2.5){\bf{1}};

        \node[above left] at (-1.5,1.5){\bf{0}};
       \node[above right] at (-0.5,1.5){\bf{0}};
       \node[below left] at (-0.5,0.5){\bf{0}};
       \node[below left] at (-1,1){\bf{0}};
        \node[above left] at (-1,2){\bf{0}};
        \node[above left] at (-0.5,2.5){\bf{0}};

\end{tikzpicture}

\caption{By inspecting formulae (\ref{weightexplicitdisplay1}), (\ref{weightexplicitdisplay2}) for the integrand we make the following observations (where without loss of generality we assume that $\beta>2$ since the integrand is uniformly bounded for $\beta \le 2$ from Lemma \ref{WeightBoundedness}). We have singularities, each of order $\left(1-\frac{2}{\beta}\right)$, when any of the coordinates of the centre row (where the two interlacing arrays are joined) are either $0$ or $1$, depicted as blue particles in the figure. Moreover, for any pair of coordinates on the same row which coalesce, depicted as dashed edges in the figure, we have a factor vanishing at order $2\left(1-\frac{2}{\beta}\right)$. Finally, for any pair of coordinates on two consecutive rows which coalesce, depicted as solid edges in the figure, we have a singularity of order $\left(1-\frac{2}{\beta}\right)$. Then, to compute the order of the singularity at a point we simply add up the number of blue particles to the number of solid edges and subtract twice the number of dashed edges and then finally multiply the result by $\left(1-\frac{2}{\beta}\right)$. Thus, the order of the singularity at the point $\mathfrak{x}(3;1)$ is $\left[2\cdot 1+4\cdot 1\right]  \left(1-\frac{2}{\beta}\right)= 6 \left(1-\frac{2}{\beta}\right)$ and similarly the order of the singularity at $\mathfrak{x}(4;1)$ is $\left[12\cdot 1+4\cdot 1-2\cdot 2\right]  \left(1-\frac{2}{\beta}\right)= 12 \left(1-\frac{2}{\beta}\right)$. }\label{Figure3}

\end{figure}
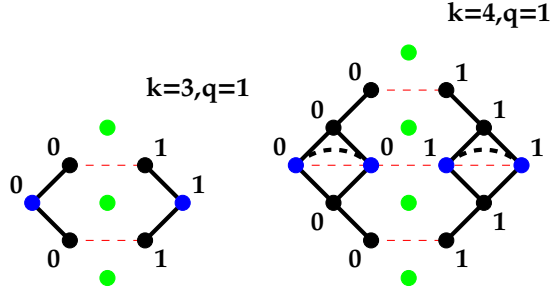

We now give some informal heuristics in support of $\mathcal{A}(k;q)=(0,2kq^2)$ and intuition behind the definition of $\mathfrak{x}(k;q)$. Observe that one of the main complications in studying the finiteness of (\ref{ConstIntegralRep}) compared to (\ref{Integralk=2}), in addition to the obvious difficulty of having to analyse an integral over the whole array instead of a single row, is that singularities can arise not only when the coordinates of a point are $0$'s and $1$'s. Nevertheless, $0$ and $1$ are distinguished points in that we have some extra singular factors there corresponding to the coordinates of the centre row, see for example Figure \ref{Figure3}. Intuitively then, in order to obtain as strict of a restriction on $\beta$ as possible, one would like to look at singularities at points having more $0$'s and $1$'s as coordinates and indeed checking what happens in a number of different cases indicates that it is actually best to only use $0$'s and $1$'s. However, the fact that one needs to take into account the sum constraints complicates things and proving this claim rigorously appears quite messy and we do not pursue it further in this paper\footnote{It is possible to prove the claim for the smallest possible values of $k,q$ by brute force computations of all the possible cases but also for $k=2$ and general $q$, in which case the desired final result $\mathcal{A}(2;q)=(0,4q^2)$ is of course already known from Lemma \ref{LemmaIntegrabilityk=2}.}. 

Assuming this unproven heuristic, it would then be possible to argue by direct computations that it is better (in that we get a stricter restriction for $\beta$) to look at a singularity at a point that involves all of the $(k-1)$ sum constraints. Due to the ordering of the coordinates on individual rows and the interlacing this uniquely identifies $\mathfrak{x}(k;q)$ as the point  with the minimal number of variables which achieves this. Finally, by some more calculations it is possible to prove that looking at a singularity at a point involving any additional coordinates to the ones already defining $\mathfrak{x}(k;q)$ will not give a stricter restriction on $\beta$, which would then imply that $\mathcal{A}(k;q)=(0,2kq^2)$.

To conclude, simply putting everything together establishes Proposition \ref{PropAsymptotics2}.

\begin{proof}[Proof of Proposition \ref{PropAsymptotics2}]
The statement follows by combining Corollary \ref{CorComparison} and Lemmas \ref{ExplicitExpressions}, \ref{LemmaIntegrabilityk=2} and \ref{LemmaSupremum} above.
\end{proof}

\subsection{On $\mathfrak{c}^{(\beta)}(2;q)$ and integrable systems}\label{IntSys}

We discuss in some more detail the integral expression of the leading order coefficient in the asymptotics for the special case $k=2$. By using Lemma \ref{ExplicitExpressions} and writing the sum constraint as a Fourier integral we have:
\begin{align*}
\mathfrak{c}^{(\beta)}(2;q)&=\prod_{M=1}^{2q}\frac{\Gamma\left(\frac{2}{\beta}\right)}{\Gamma\left(M\frac{2}{\beta}\right)^2}\int_{\mathsf{W}_+^{(2q)}\cap[0,1]^{2q},\  \sum_{i=1}^{2q}x_i=q}\prod_{1\le i<j \le 2q}(x_i-x_j)^{\frac{4}{\beta}}\prod_{i=1}^{2q}\left[x_i\left(1-x_i\right)\right]^{\frac{2}{\beta}-1}dx\\
&=\prod_{M=1}^{2q}\frac{\Gamma\left(\frac{2}{\beta}\right)}{\Gamma\left(M\frac{2}{\beta}\right)^2}\int_{-\infty}^{\infty}dse^{2\pi \i s q}\int_{\mathsf{W}_+^{(2q)}\cap[0,1]^{2q}}\prod_{1\le i<j \le 2q}(x_i-x_j)^{\frac{4}{\beta}}\prod_{i=1}^{2q}e^{-2\pi \i sx_i}\left[x_i\left(1-x_i\right)\right]^{\frac{2}{\beta}-1}dx.
\end{align*}

For $\beta=2$, by the Andreif identity, the inner $2q$-dimensional integral is a Hankel determinant corresponding to a certain special weight and is known to have a representation in terms of a particular case of the Painlev\'e V equation: 
\begin{align*}
\left(t\frac{d^2}{dt^2}\sigma_{2q}(t)\right)^2=\left(\sigma_{2q}(t)+\left(4q-t\right)\frac{d}{dt}\sigma_{2q}(t)\right)^2-4\left(\frac{d}{dt}\sigma_{2q}(t)\right)^2\left((2q)^2-\sigma_{2q}(t)+t\frac{d}{dt}\sigma_{2q}(t)\right),
\end{align*}
see \cite{BasorPainleveV}, \cite{BasorGeRubinstein}, \cite{AssiotisKeating} for the precise statement.

It is plausible that a connection to integrable systems exists for other values of $\beta$ as well, especially for the COE and CSE cases, namely for $\beta=1$ and $\beta=4$ (from Lemma \ref{LemmaIntegrabilityk=2} we need to restrict to $q>1$ for $\mathfrak{c}^{(4)}(2;q)$ to be finite), and the formula we give above could be used as a starting point for such an investigation (as the special case of this formula for $\beta=2$ was used in \cite{BasorGeRubinstein}). In particular, for $\beta=4$ we see that the inner integral can in fact be written as a Pfaffian by making use of de Bruijn's formula, see \cite{DeBruijn}. 

Finally, for $\beta=2$ and $k\ge 3$ the integral expression for $\mathfrak{c}^{(2)}(k;q)$ can be somewhat simplified since it is possible to compute the intermediate integrals between two consecutive sum constraints in terms of spline functions \cite{CurrySchoenberg}, \cite{OlshanskiProjections}, see Section 4 of \cite{AssiotisKeating} for more details. It is unclear whether an analogous simplification exists for general $\beta$.

\bigskip
\noindent
{\sc School of Mathematics, University of Edinburgh, James Clerk Maxwell Building, Peter Guthrie Tait Rd, Edinburgh EH9 3FD, U.K.}\newline
\href{mailto:theo.assiotis@ed.ac.uk}{\small theo.assiotis@ed.ac.uk}


\begin{thebibliography}{}

      \bibitem{ArguinBeliusBourgade}
           {\sc L-P. Arguin, D. Belius, P. Bourgade,}
          \textit{Maximum of the Characteristic Polynomial of Random Unitary Matrices}, Communications in Mathematical Physics, \textbf{349},  703–751, (2017).

\bibitem{AssiotisBaileyKeating}
      {\sc T. Assiotis, E.C. Bailey, J.P. Keating,}
     \textit{On the moments of the moments of the characteristic polynomials of Haar distributed symplectic and orthogonal matrices}, to appear Annales de l'Institut Henri Poincare D, available from arXiv:1910.12576, (2019).


\bibitem{AssiotisKeating}
      {\sc T. Assiotis, J.P. Keating,}
     \textit{Moments of moments of characteristic polynomials of random unitary matrices and lattice point counts}, to appear at RMTA, doi.org/10.1142/S2010326321500192, available from arXiv:1905.06072.
     
 \bibitem{AssiotisNajnudel}
      {\sc T. Assiotis, J. Najnudel,}
     \textit{The boundary of the orbital beta process}, to appear Moscow Math. J., available from arXiv:1905.08684, (2019).    


  \bibitem{BaileyKeating}
      {\sc E.C. Bailey, J.P. Keating,}
     \textit{On the moments of the moments of the characteristic polynomials of random unitary matrices}, Communications in Mathematical Physics, \textbf{371}, 689-726, (2019).
     
     \bibitem{BaileyKeatingZeta}
      {\sc E.C. Bailey, J.P. Keating,}
     \textit{On the moments of the moments of $\zeta \left(\frac{1}{2}+\i t\right)$}, to appear Journal of Number Theory, available from arxiv:2006.04503, (2020).  
     
     
     
      \bibitem{BasorPainleveV}
                    {\sc E. Basor, Y. Chen, T. Ehrhardt,}
                   \textit{Painleve V and time-dependent Jacobi polynomials}, Journal of Physics A: Mathematical and Theoretical, \textbf{43},  (2010). 
     
   \bibitem{BasorGeRubinstein}
               {\sc E. Basor, F. Ge, M.O. Rubinstein,}
              \textit{Some multidimensional integrals in number theory and connections with the Painleve V equation}, Journal of Mathematical Physics, \textbf{59}, Issue 9, (2018).
              
         \bibitem{Berestycki}
               {\sc N. Berestycki,}
              \textit{An elementary approach to Gaussian multiplicative chaos}, Electr. Comm. Probab., \textbf{22}, no. 27, 1-12, (2017).        
              
              
              
     \bibitem{BourgadeNikeghbaliRouault}
              {\sc P. Bourgade, A. Nikeghbali, A. Rouault,}
             \textit{Circular Jacobi ensembles and deformed Verblunsky coefficients}, International Mathematics Research Notices, no. 23, 4357-4394, (2009).   
     
     
       \bibitem{BumpGamburd}
           {\sc D. Bump, A. Gamburd,}
          \textit{On the Averages of Characteristic Polynomials From Classical Groups}, Communications in Mathematical Physics, 265, Issue 1, 227-274, (2006).
          
        \bibitem{ChhaibiMadauleNajnudel}
                  {\sc R. Chhaibi, T. Madaule, J. Najnudel,}
                 \textit{On the maximum of the C$\beta$E field}, Duke Mathematical Journal, Vol. 167, No. 12, 2243-2345, (2018).      
                 
                 
           \bibitem{ChhaibiNajnudel}
                  {\sc R. Chhaibi, J. Najnudel,}
                 \textit{On the circle $GMC^{\gamma}=\underset{\leftarrow}{\lim}C\beta E_n$ for $\gamma=\sqrt{\frac{2}{\beta}} (\gamma\le 1)$}, available from https://arxiv.org/abs/1904.00578, (2019).             
          
          
      \bibitem{ClaeysKrasovsky}
                  {\sc T. Claeys, I. Krasovsky,}
                 \textit{Toeplitz determinants with merging singularities}, Duke Mathematical Journal, Vol. 164, No. 15, 2897-2987, (2015).    
                 
          
    
  \bibitem{Autocorrelation}
             {\sc J.B. Conrey, D.W. Farmer, J.P. Keating, M.O. Rubinstein, N.C. Snaith,}
            \textit{Autocorrelation of random matrix polynomials}, Communications in Mathematical Physics, 237, Issue 3, 365-395, (2003).
            
            
   \bibitem{CuencaOrbital}
    {\sc C. Cuenca,}
   \textit{Universal Behavior of the Corners of Orbital Beta Processes}, International Mathematics Research Notices, rnz330, (2019).          
            
   
   \bibitem{CurrySchoenberg}
     {\sc H.B. Curry, I.J. Schoenberg,}
      \textit{On Polya frequency functions IV: The fundamental spline functions and their limits}, Journal d' Analyse Mathematique, \textbf{17}, 71-107, (1966).  
      
      
      \bibitem{DeBruijn}
     {\sc N.G. de Bruijn,}
      \textit{On some multiple integrals involving determinants}, J. Indian Math. Soc., \textbf{19}, 133-151, (1955).     
        
      \bibitem{Fahs}
         {\sc B. Fahs,}
       \textit{Uniform asymptotics of Toeplitz determinants with Fisher-Hartwig singularities}, available from arxiv:1909.07362, (2019).  
       
        \bibitem{ForkelKeating}
         {\sc J. Forkel, J. P. Keating,}
       \textit{The Classical Compact Groups and Gaussian Multiplicative Chaos}, available from arxiv:2008.07825, (2020).    
    
      \bibitem{Forrester}
         {\sc P. J. Forrester ,}
       \textit{Log-gases and random matrices}, Princeton University Press, (2010).
       
         \bibitem{FyodorovBouchaud}
       {\sc Y. V. Fyodorov, J-P. Bouchaud,}
      \textit{Freezing and extreme value statistics in a Random Energy Model with logarithmically correlated potential}, Journal of Physics A: Mathematical and Theoretical,  Volume 41, Number 37, 372001, (2008). 
        
       
    
           \bibitem{GnutzmannFyodorovKeating}
       {\sc Y. V. Fyodorov, S. Gnutzmann, J.P. Keating,}
      \textit{Extreme values of CUE characteristic polynomials: a numerical study}, Journal of Physics A: Mathematical and Theoretical, Vol. 51, Issue 46, (2018).    
    
    
    \bibitem{FyodorovKeating}
       {\sc Y. V. Fyodorov, J.P. Keating,}
      \textit{Freezing transitions and extreme values: random matrix theory, and disordered landscapes}, Philosophical Transactions of the Royal Society A: Mathematical, Physical and Engineering Sciences, \textbf{372}, (2014). 
      
      
      
   \bibitem{GorinMarcus}
       {\sc V. Gorin, A. Marcus,}
      \textit{Crystallization of random matrix orbits}, International Mathematics Research Notices, \textbf{3}, 883-913, (2020).  
         
         
        \bibitem{HKO}
           {\sc C. P. Hughes, J.P. Keating, N. O'Connell,}
          \textit{On the Characteristic Polynomial of a Random Unitary Matrix}, Communications in Mathematical Physics, 220, no. 2, 429-451, (2001).    
         
         
         
         
       \bibitem{JiangMatsumoto}
                  {\sc T. Jiang, S. Matsumoto,}
                 \textit{Moments of traces of circular beta-ensembles}, Annals of Probability, Vol. 43, No. 6, 3279-3336, (2015).    
                 
                 
            \bibitem{Johansson}
                  {\sc K. Johansson,}
                 \textit{On Szeg\H{o}'s asymptotic formula for Toeplitz determinants and generalizations,} Bull. Sci. Math. (2) 112, No. 3, 257-304, (1988).                
                 
    

    
    \bibitem{KeatingSnaith}
               {\sc J.P. Keating, N.C. Snaith,}
              \textit{Random matrix theory and $\zeta \left(\frac{1}{2}+\i t\right)$}, Communications in Mathematical Physics, 214, Issue 3, 57-89, (2000).  
              
       \bibitem{KeatingRodgersRodittyGershonRudnick}
                  {\sc J.P. Keating, B. Rodgers, E. Roditty-Gershon, Z. Rudnick,}
                 \textit{Sums of divisor functions in $\mathbb{F}_q[t]$ and matrix integrals}, Mathematische Zeitschrift, \textbf{288}, no. 1-2, 167-198, (2018).   
      
            \bibitem{KeatingWong}
                  {\sc J.P. Keating, M.D. Wong,}
                 \textit{On the critical-subcritical moments of moments of random characteristic polynomials: a GMC perspective}, avalaible from           arxiv:2012.15851, (2020).               
                 
                 
                 
   \bibitem{KillipNenciu}
       {\sc R. Killip, I. Nenciu,}
      \textit{Matrix models for circular ensembles}, International Mathematics Research Notices, \textbf{50}, 2665–2701, (2004).  
      
      
         \bibitem{Lambert}
   {\sc G. Lambert,}
         \textit{Mesoscopic central limit theorem for the circular $\beta$-ensembles and applications}, Electronic Journal of Probability, \textbf{26}, no. 7, 1-33, (2021).          
                         
                 
                 
    \bibitem{Macdonald}
    {\sc I. G. Macdonald,}
   \textit{Symmetric Functions and Hall Polynomials}, Second Edition, Oxford University Press, (1995).               
    
    \bibitem{Matsumoto}
                  {\sc S. Matsumoto,}
                 \textit{Moments of characteristic polynomials for compact
symmetric spaces and Jack polynomials}, J. Phys. A: Math. Theor. \textbf{40}, 13567-13586, (2007).         


    \bibitem{NajnudelPaquetteSimm}
                  {\sc J. Najnudel, E. Paquette, N. Simm,}
                 \textit{Secular Coefficients and the Holomorphic Multiplicative Chaos}, available from https://arxiv.org/abs/2011.01823, (2020).     
                 
       \bibitem{GMCL1}
   {\sc M. Nikula, E. Saksman, C. Webb,}
         \textit{Multiplicative chaos and the characteristic polynomial of the CUE: The $ L^1$-phase}, Trans. Amer. Math. Soc. \textbf{373} , 3905-3965, (2020).                       
                 
                 
                 
                 
   \bibitem{OkounkovOlshanskiJack}
  {\sc A. Okounkov, G. Olshanski,}
 \textit{Asymptotics for Jack polynomials as the number of variables goes to infinity}, International Mathematics Research Notices,No. 13, 641-682, (1998).              
                 
                 
   \bibitem{OlshanskiProjections}
   {\sc G. Olshanski,}
         \textit{Projections of orbital measures, Gelfand-Tsetlin polytopes and splines}, Journal of Lie Theory,Vol. 23, No.4, 1011-1022, (2013).  
         
    \bibitem{PaquetteZeitouni}
       {\sc E. Paquette, O. Zeitouni,}
      \textit{The Maximum of the CUE Field}, International Mathematics Research Notices, \textbf{16}, 5028–5119, (2018).    
      
       \bibitem{Remy}
       {\sc G. Remy,}
      \textit{The Fyodorov–Bouchaud formula and Liouville conformal field theory}, Duke Math. J., 169, (1), 177-211, (2020).             
      
         
         
    \bibitem{Stanley}
   {\sc R. Stanley,}
         \textit{Some combinatorial properties of Jack symmetric functions}, Advances in Mathematics,Vol. 77, No.1, 76-115, (1989).  
         
       \bibitem{Webb}
   {\sc C. Webb,}
         \textit{The characteristic polynomial
of a random unitary matrix and Gaussian multiplicative chaos - The $L^2$-phase}, Electronic Journal of Probability, \textbf{20}, no. 104, 1-21, (2015).       


\end{thebibliography}
\end{document}